\documentclass[11pt, a4paper]{article}

\usepackage{amssymb,amsmath, wasysym, graphicx, epsfig, setspace, wrapfig, comment}
\usepackage{latexsym,dsfont,amsfonts,amsmath,amssymb,amsthm}
\usepackage{pstricks,epsfig,psfrag,color}
\usepackage{bbm}
\usepackage[show]{ed}
\usepackage{mathrsfs}
\usepackage{soul}
\usepackage{mathrsfs}   
\usepackage{tikz}
\usepackage{url}
\usepackage{xfrac}
\usepackage{algpseudocode,algorithm} 
\algdef{SE}[DOWHILE]{Do}{doWhile}{\algorithmicdo}[1]{\algorithmicwhile\ #1}

\usepackage{mathtools}

\newcommand{\pd}[3]{\frac{\partial ^{#1} #2}{\partial #3}}

\newcommand{\nrm}[2]{\ensuremath{\|#1\|_{#2}}}

\newcommand{\pdy}{\partial_{\bsy}}
\newcommand{\dist}{\mathrm{dist}}
\newcommand{\cost}{\mathrm{cost}}

\newcommand{\diam}{\mathrm{diam}}

\newcommand{\Df}{\ensuremath{D_\mathrm{f}}}

\newcommand{\N}[0]{\mathbb{N}}

\newcommand{\R}[0]{\mathbb{R}}

\newlength\figureheight
\newlength\figurewidth

\makeatletter
\newcommand{\BIG}{\bBigg@{3}}
\newcommand{\vast}{\bBigg@{4}}
\newcommand{\Vast}{\bBigg@{5}}
\makeatother

\newcommand{\Qhat}[0]{\ensuremath{\widehat{Q}}}
\newcommand{\Ehat}[0]{\ensuremath{\widehat{E}}}
\newcommand{\Vhat}[0]{\ensuremath{\widehat{\bbV}}}
\newcommand{\Uni}[0]{\ensuremath{\mathrm{U}}}

\newcommand{\Qml}[0]{\ensuremath{Q^\mathrm{ML}}}
\newcommand{\Qhatml}[0]{\ensuremath{\Qhat^\mathrm{ML}}}
\newcommand{\Qhattg}[0]{\ensuremath{\Qhat^\mathrm{TG}}}
\newcommand{\Csetup}[0]{\ensuremath{\calC^\mathrm{setup}}}
\newcommand{\Csolve}[0]{\ensuremath{\calC^\mathrm{solve}}}

\newtheorem{theorem}{Theorem}[section]
\newtheorem{lemma}{Lemma}[section]

\newtheorem{proposition}[theorem]{Proposition}
\newtheorem{definition}{Definition}[section]
\newtheorem{assumption}{Assumption A$\!\!\!$}
\theoremstyle{definition}

\theoremstyle{plain}

\numberwithin{equation}{section}

\newcommand{\sigdiff}[0]{\ensuremath{\sigma_\mathrm{diff}}}
\newcommand{\sigabs}[0]{\ensuremath{\sigma_\mathrm{abs}}}

\newcommand{\bsbeta}{{\boldsymbol{\beta}}}
\newcommand{\bsDelta}{{\boldsymbol{\Delta}}}

\newcommand{\bsa}{{\boldsymbol{a}}}
\newcommand{\bsb}{{\boldsymbol{b}}}

\newcommand{\bsgamma}{{\boldsymbol{\gamma}}}

\newcommand{\bst}{{\boldsymbol{t}}}
\newcommand{\bsu}{{\boldsymbol{u}}}

\newcommand{\bsx}{{\boldsymbol{x}}}

\newcommand{\bsy}{{\boldsymbol{y}}}

\newcommand{\bsz}{{\boldsymbol{z}}}

\newcommand{\bszero}{{\boldsymbol{0}}}
\newcommand{\bsone}{{\boldsymbol{1}}}

\newcommand{\rd}{\,\mathrm{d}}

\newcommand{\bbE}{\mathbb{E}}
\newcommand{\bbV}{\mathbb{V}}
\newcommand{\calA}{\mathcal{A}}
\newcommand{\calB}{\mathcal{B}}
\newcommand{\calC}{\mathcal{C}}

\newcommand{\calI}{\mathcal{I}}

\newcommand{\calL}{\mathcal{L}}
\newcommand{\calM}{\mathcal{M}}

\newcommand{\calG}{\mathcal{G}}

\newcommand{\calP}{\mathcal{P}}
\newcommand{\calW}{\mathcal{W}}

\newcommand{\calT}{\mathcal{T}}

\newcommand{\Forall}{\quad\text{ for all }}

\def\R{\mathbb{R}}

\newcommand{\setu}{{\mathrm{\mathfrak{u}}}}

\newcommand{\mask}[1]{{}}

\definecolor{darkred}{RGB}{139,0,0}
\definecolor{darkgreen}{RGB}{0,100,0}
\definecolor{darkmagenta}{RGB}{170,0,120}
\definecolor{darkpurple}{RGB}{110,0,180}
\definecolor{darkblue}{RGB}{40,0,200}
\definecolor{darkbrown}{rgb}{0.75,0.40,0.15}
\definecolor{grey}{RGB}{59,61,63}

\newcommand{\be}{\begin{equation}}
\newcommand{\ee}{\end{equation}}
\newcommand{\bea}{\begin{eqnarray}}
\newcommand{\eea}{\end{eqnarray}}
\newcommand{\beas}{\begin{eqnarray*}}
\newcommand{\eeas}{\end{eqnarray*}}

\def\r2p{{\sqrt{2\pi}}}

\newcommand*{\amin}{\ensuremath{a_{\min}}}
\newcommand*{\amax}{\ensuremath{a_{\max}}}

\newcommand*{\ubar}{\ensuremath{\overline{u}}}
\newcommand*{\lambdaunder}{\ensuremath{\underline{\lambda}}}
\newcommand*{\lambdaover}{\ensuremath{\overline{\lambda}}}

\newcommand{\hsuff}{\ensuremath{\overline{h}}}

\newcommand{\Ws}[0]{\ensuremath{\calW_{s, \bsgamma}}}

\newcommand{\bigO}{\mathcal{O}}

\graphicspath{{.}{./Figures/}}

\title{Multilevel quasi-Monte Carlo for random elliptic
  eigenvalue problems II: Efficient algorithms and numerical results}

\date{\today}

\makeatletter
\let\@fnsymbol\@arabic
\makeatother

\author{Alexander D. Gilbert\footnotemark[1]
				\and
             Robert Scheichl\footnotemark[2]
             }

\begin{document}

\maketitle

\footnotetext[1]{School of Mathematics and Statistics, University of New South Wales, 
                           Sydney NSW 2052, Australia.\\
                           \texttt{alexander.gilbert@unsw.edu.au}
                           }
\footnotetext[2]{Institute for Applied Mathematics \& Interdisciplinary Centre  for Scientific Computing,
								 Universit\"at Heidelberg, 
                          69120 Heidelberg, Germany and
                          Department of Mathematical Sciences, University of Bath, Bath BA2 7AY UK.\\
                          \texttt{r.scheichl@uni-heidelberg.de}
                          }

\begin{abstract}
{Stochastic PDE eigenvalue problems often arise in the field of uncertainty quantification,
whereby one seeks to quantify the uncertainty in an eigenvalue, or its eigenfunction. 
In this paper we present an efficient multilevel quasi-Monte Carlo (MLQMC) algorithm
for computing the expectation of the smallest eigenvalue of an elliptic eigenvalue
problem with stochastic coefficients.
Each sample evaluation requires the solution of a PDE eigenvalue problem, and so
tackling this problem in practice is notoriously computationally difficult.
We speed up the approximation of this expectation in four ways: 
we use a multilevel variance reduction scheme to spread the work over a hierarchy of
FE meshes and truncation dimensions; 
we use QMC methods to efficiently compute the expectations on each level;
we exploit the smoothness in parameter space and reuse the
eigenvector from a nearby QMC point to reduce the number of iterations 
of the eigensolver; and  
we utilise a two-grid discretisation scheme to obtain the eigenvalue on the fine mesh with a 
single linear solve.
The full error analysis of a basic MLQMC algorithm is given in the
companion paper [Gilbert and Scheichl, 2022], and so in this paper we
focus on how to further improve the efficiency and provide theoretical 
justification for using nearby QMC points and two-grid methods.
Numerical results are presented that show the efficiency of our algorithm, and
also show that the four strategies we employ are complementary.}
\end{abstract} 

\section{Introduction}

In this paper we develop efficient methods for computing the expectation
of an eigenvalue of the stochastic eigenvalue problem (EVP)
\begin{align}
\label{eq:evp}
\nonumber 
-\nabla\cdot\big(a(\bsx, \bsy)\,\nabla u(\bsx, \bsy)\big) 
+ b(\bsx, \bsy)\,u(\bsx, \bsy)
&= 
\lambda(\bsy) \,c(\bsx,\bsy ) \,u(\bsx, \bsy), \ \  &\text{for } \bsx \in D,
\\
u(\bsx, \bsy) \,&=\, 0 \quad &\text{for }\bsx \in \partial D,
\end{align}
where the differential operator $\nabla$ is with respect to $\bsx$, 
which  belongs to the \emph{physical domain} $D \subset \R^d$ for $d = 1, 2, 3$.
Randomness is incorporated into the PDE \eqref{eq:evp} through the dependence
of the coefficients $a$, $b$ on the {stochastic parameter} 
$\bsy\,=\, (y_j)_{j \in \N} $,
which is a countably infinite-dimensional vector with i.i.d. uniformly distributed entries:
$y_j \sim \Uni[-\frac{1}{2}, \frac{1}{2}]$ for $j \in \N$.
The whole stochastic parameter domain is denoted by 
$\Omega \coloneqq [-\tfrac{1}{2}, \tfrac{1}{2}]^\N$.

The study of stochastic PDE problems is motivated by 
applications in uncertainty quantification---where one is interested in quantifying
how uncertain input data
affect model outputs. In the case of \eqref{eq:evp}
the uncertain input data are the coefficients $a$ and $b$, and the 
outputs of interest are the eigenvalue $\lambda(\bsy)$ and its corresponding
eigenfunction $u(\bsx, \bsy)$, 
which are now also stochastic objects.
As such, to quantify uncertainty we would like to compute statistics of the eigenvalue 
(or eigenfunction),
and in particular, in this paper we compute the expectation
of the smallest eigenvalue $\lambda$ with respect to the countable product of
uniform densities. This is formulated as the infinite-dimensional integral
\[
\bbE_\bsy[\lambda] \,=\, \int_{\Omega} \lambda(\bsy) \, \rd\bsy
\,\coloneqq\, \lim_{s \to \infty} \int_{[-\frac{1}{2}, \frac{1}{2}]^s}
\lambda(y_1, y_2, \ldots, y_s, 0, 0 \ldots)\, \rd y_1 \rd y_2 \cdots \rd y_s\,.
\]

EVPs corresponding to differential operators
appear in many applications from engineering and the
physical sciences, e.g., structural vibration analysis \cite{Thom81}, 
nuclear reactor criticality \cite{DH76,JC13} or 
photonic crystal structures \cite{D99,K01,NS12}.
In addition, stochastic EVPs, such as \eqref{eq:evp}, have recently 
garnered more interest
due to the desire to quantify the uncertainty present in such applications 
\cite{ShinAst72,AI10,W10,AEHW12,QiuLyu20}.
Thus, significant development has recently also gone into efficient 
numerical methods for 
tackling such stochastic EVPs in practice, 
the most common being Monte Carlo \cite{ShinAst72}, 
stochastic collocation \cite{AS12} and stochastic Galerkin/polynomial chaos
methods \cite{GhanGhosh07,W10}. 
The latter two classes perform poorly for high-dimensional problems, 
so in order to handle the high-dimensionality of the parameter
space, sparse and low-rank versions of those methods
have been developed, see e.g., \cite{AS12,HakKaarLaak15,ElmSu19}.
Furthermore, to improve upon classical Monte Carlo, while still performing well
in high dimensions, the present authors with their colleagues have analysed the use of
quasi-Monte Carlo methods \cite{GGKSS19,GGSS20}.

In practice, for each parameter value $\bsy \in \Omega$
the elliptic EVP \eqref{eq:evp} must be solved numerically, 
which we do here by the finite element (FE) method, see, e.g., \cite{BO91}. 
First, the spatial domain is discretised by a family of triangulations $\{\mathscr{T}_h\}_{h > 0}$
indexed by the meshsize $h> 0$, and then \eqref{eq:evp} is solved on the finite-dimensional
FE space corresponding to $\mathscr{T}_h$.
This leads to a large, sparse, symmetric matrix EVP, which is typically solved by
an iterative method (such as Rayleigh quotient iteration or the Lanczos algorithm, 
see, e.g., \cite{Parlett80,Saad11}), requiring several solves of a linear system.
To speed up the solution of each EVP we use
the {\em accelerated two-grid method} developed independently in
\cite{HuCheng11,HuCheng15_corr} and \cite{YangBi11}.
In particular, to obtain an eigenvalue approximation corresponding to
a ``fine'' mesh $\mathscr{T}_h$,
one first solves the FE EVP on a ``coarse'' mesh $\mathscr{T}_H$, 
with $H \gg h$, to obtain a coarse eigenpair $(\lambda_H, u_H)$.
An eigenvalue approximation $\lambda^h$ on the fine grid $\mathscr{T}_h$
is then obtained by performing a single step of shifted inverse iteration
with shift $\lambda_H$ and start vector $u_H$.
Typically, the fast convergence rates of FE methods and of shifted
inverse iteration for eigenvalue problems allow for a very large
difference between the coarse and fine meshsize, e.g., for piecewise 
polynomial FE spaces it is sufficient to take $H \eqsim h^{1/4}$,
so that the cost of the two-grid method essentially reduces to the
cost of a single linear solve on the fine mesh. 
Multilevel sampling schemes also exploit a hierarchy of FE meshes, 
and so the two-grid method very naturally fits into this framework.

The multilevel Monte Carlo (MLMC) method \cite{Giles08a,Hein01} is a variance reduction 
scheme that has achieved great success when applied to problems including path integration
 \cite{Hein01},
stochastic differential equations \cite{Giles08a} and also stochastic PDEs
\cite{BarSchwZol11,ClifGilSchTeck11}.
For stochastic PDE problems, it is based on a hierarchy
of $L + 1$ increasingly fine FE meshes $\{\mathscr{T}_{h_\ell}\}_{\ell = 0}^L$
(i,.e., the meshwidths are decreasing $h_0 > h_1 > \cdots > h_L> 0$),
and an increasing sequence of truncation dimensions $s_0 < s_1 < \cdots < s_L < \infty$
of the infinite-dimensional parameter domain $\Omega$.
For the eigenvalue problem \eqref{eq:evp},
letting the truncation-FE approximation on level $\ell$ be denoted by
$\lambda_\ell \coloneqq \lambda_{h_\ell, s_\ell}$,  
the key idea is to write the expectation on the desired finest level $L$
as a telescoping sum of differences:
\begin{equation}
\label{eq:tele_sum}
\bbE_\bsy [ \lambda_L] \,=\, \bbE_\bsy[\lambda_0] + \sum_{\ell = 1}^L \bbE_\bsy[\lambda_\ell - \lambda_{\ell - 1}],
\end{equation}
and then compute each expectation $\bbE_\bsy [\lambda_\ell - \lambda_{\ell - 1}]$
by an independent MC approximation. As $\ell \to \infty$, provided
$h_\ell \to 0$ and $s_\ell \to \infty$, we have $\lambda_\ell \to \lambda$ and hence also
$\lambda_\ell - \lambda_{\ell - 1} \to 0$. Thus the variance on each level decreases,
and so less samples will be needed on the finer levels.
In this way, the MLMC method achieves a significant cost reduction by spreading 
the work across the hierarchy of levels,
instead of performing all evaluations on the finest level $L$.
For any linear functional $\calG$, 
we can write a similar telescoping sum for
$\bbE_\bsy[\calG(u_L)]$, where we define $u_\ell \coloneqq u_{h_\ell, s_\ell}$.
The smallest eigenvalue is simple, therefore we can ensure that the corresponding 
eigenfunction is unique by normalising it and choosing the sign consistently. 
Similarly, we normalise each approximation $u_\ell$ and choose the sign to match
the eigenfunction, which ensures they are also well defined.
Our method can also be applied to approximate the expectation of other simple eigenvalues higher up the 
spectrum, without any essential modifications.
If the eigenvalue in question is well-separated from the rest of the spectrum (uniformly in $\bsy$) 
then our analysis can also be extended in a straightforward way.
However, for simplicity and clarity of presentation in this paper we focus on the smallest eigenvalue.

Quasi-Monte Carlo methods are equal-weight quadrature rules that are tailored
to efficiently approximate high-dimensional integrals, see, e.g., \cite{DKS13,DP10}.
In particular, by deterministically choosing well-distributed
quadrature points, giving preference to more important dimensions, QMC
rules can be constructed such that the error converges faster than for
MC  methods, whilst still being independent of dimension.
Using a QMC rule to approximate the expectation on each level in
\eqref{eq:tele_sum} instead of Monte Carlo gives a Multilevel quasi-Monte Carlo (MLQMC) method.
MLQMC methods were first developed in \cite{GilesWater09} 
for option pricing, and since then have also had great success for UQ in
stochastic PDE problems  \cite{GS20a,KSS15,KSSSU17}.
The gains are complementary,
so that for several problems MLQMC methods can be shown to result in faster convergence than 
either MLMC methods or single level QMC approximations.

In this paper, we present an efficient MLQMC method for computing the expectation
of the smallest eigenvalue of stochastic eigenvalue problems of the form \eqref{eq:evp}.
We employ four complementary strategies:
1) we use the ML strategy to reduce the variance and spread the work across
a hierarchy of FE meshes and truncation dimensions;
2) we use QMC methods to compute the expectation on each level
more efficiently;
3) we use the two-grid method for eigenvalue problems 
\cite{HuCheng11,HuCheng15_corr,YangBi11}
to compute the eigenpair on finer grids using an eigensolve on a very coarse grid
followed by a single linear solve; and
4) we reuse the eigenvector corresponding to a nearby QMC point as
the starting vector for the Rayleigh quotient algorithm to solve each eigenvalue problem.

The focus of this paper is on developing the above practical strategies to give
an efficient MLQMC method. A rigorous analysis of MLQMC methods for
the stochastic EVP \eqref{eq:evp} is the focus of a separate paper \cite{GS20a}.
However, we do give a theoretical justification of the enhancement
strategies 3) and 4). First, we extend the two-grid method and its
analysis to stochastic EVPs, allowing also for a reduced truncation dimension on
the ``coarse grid''.
Second, we analyse the benefit of using an eigenvector corresponding to a nearby QMC
point as the starting vector for the iterative eigensolve.

The structure of the paper is as follows. In Section~\ref{sec:background}
we present the necessary background material. Then in Section~\ref{sec:2grid}
we extend the two-grid method for deterministic EVPs to stochastic
EVPs and analyse the error.
In Section~\ref{sec:mlqmc-alg} we describe a basic MLQMC algorithm, and
then outline how one can reduce the cost by using two-grid methods and
nearby QMC points.
Finally, in Section~\ref{sec:num} we present numerical results for two test problems.

\section{Mathematical background}
\label{sec:background}
In this section we briefly summarise the relevant material
on variational EVPs, two-grid FE methods and QMC methods. For further details we refer the reader
to the references indicated throughout or to \cite{GGKSS19}.

We make the following assumptions on the physical domain and on the boundedness of the
coefficients (from above and below). These ensure the well-posedness of \eqref{eq:evp} 
and admit a fast convergence rate of our MLQMC algorithm.

\begin{assumption}
\label{asm:coeff}
\hfill
\begin{enumerate}
\item $D \subset \R^d$, for $ d = 1, 2, 3$, is bounded and convex.
\item\label{itm:coeff} $a$ and $b$ are of the form 
\begin{align}
\label{eq:coeff}
a(\bsx, \bsy) = a_0(\bsx) + \sum_{j = 1}^\infty y_j a_j(\bsx)
\quad \text{and} \quad
b(\bsx, \bsy) = b_0(\bsx) + \sum_{j = 1}^\infty y_j b_j(\bsx)\,,
\end{align}
where $a_j,\ b_j \in L^\infty(D)$, for all $j \ge 0$, and $c \in L^\infty(D)$
depend on $\bsx$ but not $\bsy$.
\item\label{itm:amin} There exists $\amin > 0$ such that $a(\bsx, \bsy) \geq \amin$,
$b(\bsx, \bsy) \geq 0$ and $c(\bsx) \geq \amin$, for all $\bsx \in D$, $\bsy \in \Omega$.
\item\label{itm:summable} There exists $p \in (0, 1)$ and $q \in (0, 1)$ such that
\begin{align*}
\sum_{j = 1}^\infty \max\big(\nrm{a_j}{L^\infty},\, \nrm{b_j}{L^\infty} \big)^p < \infty
\quad \text{and} \quad
\sum_{j = 1}^\infty \nrm{\nabla a_j}{L^\infty}^q \,<\, \infty\,.
\end{align*}
\end{enumerate}
\end{assumption}

For convenience, we then let $\amax < \infty$ be such that
\begin{equation}
\label{eq:coeff_bnd}
\max \big\{ \|a(\bsy)\|_{L^\infty}, \|\nabla a (\bsy)\|_{L^\infty}, 
\|b(\bsy)\|_{L^\infty}, \|c\|_{L^\infty}\big\}
\,\leq\, \amax,
\quad \text{for all } \bsy \in \Omega.
\end{equation}

\subsection{Variational eigenvalue problems}
\label{sec:var-evp}

For the variational form of the EVP \eqref{eq:evp},
we introduce the usual function space setting for second-order elliptic PDEs:
the first-order Sobolev space of functions with zero trace is denoted
by $V \coloneqq H^1_0(D)$
and equipped with the norm $\nrm{v}{V} \coloneqq \nrm{\nabla v
}{L^2}$. Its dual space is $V^* \coloneqq H^{-1}(D)$. We will also use the Lebesgue space $L^2(D)$,
equipped with the usual inner product $\langle \cdot, \cdot \rangle_{L^2}$,
and the induced norm $\|\cdot \|_{L^2}$.

Next, for each $\bsy \in \Omega$ define the bilinear form $\calA(\bsy) : V \times V \to \R$ by 
\[
\calA(\bsy; w, v) \,\coloneqq\, 
\int_D a(\bsx, \bsy) \nabla w(\bsx) \cdot \nabla v(\bsx)\rd \bsx 
+ \int_D b(\bsx, \bsy) w(\bsx) v(\bsx) \rd \bsx,
\]
which is also an inner product on $V$ and admits the induced norm
$\nrm{v}{\calA(\bsy)} \coloneqq \sqrt{\calA(\bsy; v, v)}$.
We define also the inner product $\calM : V \times V \to \R$ by 
\[
\calM(w, v) \,\coloneqq\, \int_D c(\bsx) w(\bsx) v(\bsx) \rd \bsx\,,
\]
again with induced norm given by $\nrm{v}{\calM} \coloneqq \sqrt{\calM(v, v)}$.
Further, let $\calM(\cdot, \cdot)$ also denote the duality paring on $V \times V^*$. 

The variational form of the EVP \eqref{eq:evp} is:
Find $\lambda(\bsy) \in \R$, $u(\bsy) \in V$ such that
\begin{align}
\label{eq:var-evp}
\calA(\bsy; u(\bsy), v) \,&=\, \lambda(\bsy) \calM(u(\bsy), v)
\quad \text{for all } v \in V\,,\\
\nonumber
\nrm{u(\bsy)}{\calM} \,&=\, 1.
\end{align}
The variational EVP \eqref{eq:var-evp} is symmetric and so it is well-known
that \eqref{eq:var-evp} admits countably many, strictly positive eigenvalues, see, e.g., 
\cite{BO91}.
The eigenvalues -- labelled in ascending order, counting
multiplicities -- and the corresponding eigenfunctions are denoted by
\[
0 \,<\, \lambda_1(\bsy) \,\leq\, \lambda_2(\bsy) \,\leq\, \cdots,
\quad \text{and} \quad 
u_1(\bsy), \ u_2(\bsy),\ \ldots\,.
\]

For $\bsy \in \Omega$, we define the solution operator $T = T(\bsy) : V^* \to V$ by
\[
\calA(\bsy; T f, v) \,=\, \calM(f, v )
\quad \text{for all } v \in V.
\]
Clearly, if $\lambda(\bsy)$ is an eigenvalue of \eqref{eq:var-evp} then 
$\mu(\bsy) = 1/\lambda(\bsy)$ is an eigenvalue of $T$
and the corresponding eigenspaces are the same.

The Krein--Rutmann Theorem ensures the smallest eigenvalue is simple,
and then in \cite[Prop.~2.4]{GGKSS19}  it was shown that
the spectral gap can be bounded away from 0 independently  of $\bsy$.
That is, there exists $\rho > 0$, independent of $\bsy$, such that
\begin{equation}
\label{eq:gap}
\lambda_2(\bsy) - \lambda_1(\bsy) \,\geq\, \rho 
\quad \text{for all } \bsy \in \Omega.
\end{equation}
The eigenfunctions $\{u(\bsy)\}_{k \in \N}$ can be chosen to form a basis for 
$V$ that is orthonormal with respect to
$\calM(\cdot, \cdot)$, and hence, by \eqref{eq:var-evp}, also orthogonal with 
respect to $\calA(\bsy; \cdot, \cdot)$. For $\bsy \in \Omega$, let the eigenspace
$E(\lambda_k(\bsy))$ be the subspace spanned by all eigenfunctions
corresponding to $\lambda_k(\bsy)$, and let 
$\Ehat(\lambda_k(\bsy)) \coloneqq \{v : v \in E(\lambda_k(\bsy)), \|v\|_V = 1\}$.

Since the coefficients are uniformly bounded away from 0 and from above, 
the $\calA(\bsy)$- and $\calM$-norms are equivalent to the $V$- and $L^2$-norms, 
respectively, with
\begin{align}
\label{eq:A_equiv}
c_\calA \nrm{v}{V} \,\leq\, &\nrm{v}{\calA(\bsy)} \, \leq\, 
C_\calA \nrm{v}{V}\,,\\
\label{eq:M_equiv}
c_\calM \nrm{v}{L^2} \,\leq\, &\nrm{v}{\calM} \,\leq\, C_\calM\nrm{v}{L^2},
\end{align}
where the constants are independent of $\bsy$, see \cite[eqs. (2.7), (2.8)]{GS20a} for their
explicit values.
By $C_\mathrm{Poin} > 0$ we denote the
Poincar\'e constant, which is independent of $\bsy$ and such that
\begin{align}\label{eq:poin}
\nrm{v}{L^2(D)} \,\leq\, C_\mathrm{Poin}\nrm{v}{V},
\quad \text{for all } v \in V.
\end{align}
 
For the remainder of the paper we denote 
the smallest eigenvalue and its corresponding eigenfunction
by $\lambda = \lambda_1$ and $u = u_1$, respectively.

\subsection{Stochastic dimension truncation}

In order to evaluate the stochastic coefficients $a(\bsy)$ and $b(\bsy)$ in practice,
we must first truncate the infinite-dimensional stochastic domain $\Omega$.
This is done by choosing a finite \emph{truncation dimension} $s \in
\N$  and by setting $y_j = 0$ for all $j > s$.
We define the following notation: $\bsy_s = (y_1, y_2, \ldots, y_s)$,
\[
a^s(\bsx, \bsy) \,\coloneqq\, a_0(\bsx) + \sum_{j = 1}^s y_j a_j(\bsx)
\quad \text{and} \quad
b^s(\bsx, \bsy) \,\coloneqq\, b_0(\bsx) + \sum_{j = 1}^s y_j b_j(\bsx).
\]
In this way, the truncated coefficients $a^s(\bsy)$ and $b^s(\bsy)$ can be evaluated
in practice, since they only depend on finitely many terms.

Similarly, the truncated approximations of the eigenvalue and eigenfunction are denoted by $\lambda_s(\bsy), u_s(\bsy)$, respectively.
Defining the bilinear form $\calA_s(\bsy) : V \times V \to \R$ corresponding to the
truncated coefficients by
\begin{equation}
\label{eq:coeff_s}
\calA_s(\bsy; w, v) \,\coloneqq 
\int_D a^s(\bsx, \bsy) \nabla w(\bsx) \cdot \nabla v(\bsx) \rd \bsx
+ \int_D b^s(\bsx, \bsy)w(\bsx) v(\bsx) \rd \bsx\,,
\end{equation}
we have that $\lambda_s(\bsy), u_s(\bsy)$ satisfy
\begin{equation}
\label{eq:trunc-evp}
\calA_s(\bsy; u_s(\bsy), v) \,=\, \lambda_s(\bsy) \calM(u_s(\bsy), v),
\quad \text{for all } v \in V\,.
\end{equation}

\subsection{Finite element methods for eigenvalue problems}
\label{sec:fem}
The eigenvalue problem \eqref{eq:var-evp} will be discretised in the spatial 
domain using piecewise linear finite elements (FE).
First, we partition the spatial domain $D$ using a family of shape regular triangulations
$\{\mathscr{T}_h\}_{h > 0}$, indexed by the meshwidth 
$h = \max\{\diam(\tau) : \tau \in \mathscr{T}_h\}$.

Then, for $h > 0$ let $V_h$ be the conforming FE space
of continuous functions that are piecewise linear on the elements of the 
triangulation $\mathscr{T}_h$, and let $M_h \coloneqq \dim (V_h) < \infty$ denote
the dimension of this space.
Additionally, we assume that each mesh $\mathscr{T}_h$ is such that the
dimension of the corresponding FE space $V_h$ is
\begin{equation}
\label{eq:fe_dof}
M_h \eqsim h^{-d},
\end{equation}
which will be satisfied by quasi-uniform meshes, but also allows for locally refined meshes.

For each $\bsy \in \Omega$, the FE eigenvalue problem is:
Find $\lambda_h(\bsy) \in \R$, $u_h(\bsy) \in V_h$ such that
\begin{align}
\label{eq:fe-evp}
\calA(\bsy; u_h(\bsy), v_h) \,&=\, \lambda_h(\bsy) \calM(u_h(\bsy), v_h)
\quad \text{for all } v_h \in V_h\,,\\
\nonumber
\nrm{u_h(\bsy)}{\calM} \,&=\, 1.
\end{align}
The FE eigenvalue problem \eqref{eq:fe-evp} admits
$M_h$ eigenvalues and corresponding eigenvectors
\[
0 \,<\, \lambda_{1, h}(\bsy) \,\leq\, \lambda_{2, h}(\bsy) \,\leq\, \cdots \,\leq\, \lambda_{M_h, h}(\bsy)\,,
\quad \text{and} \quad
u_{1, h}(\bsy),\ u_{2, h}(\bsy),\ \ldots,\ u_{M_h, h}(\bsy)\,,
\]
which converge from above to the first $M_h$
eigenvalues and eigenfunctions of \eqref{eq:var-evp} as $h \to 0$,
see, e.g., \cite{BO91} or \cite{GGKSS19} for the stochastic case.

As before, let $E(\lambda_{k, h}(\bsy))$ be the eigenspace corresponding
to $\lambda_{k, h}(\bsy)$ and define 
$\Ehat(\lambda_{k, h}(\bsy)) \coloneqq \{v \in E(\lambda_{k, h}(\bsy)) : \|v\|_V = 1\}$.
If the exact eigenvalue $\lambda_k(\bsy)$ has multiplicity~$m$ (and we
assume without loss of generality that
$\lambda_k(\bsy) = \lambda_{k + 1}(\bsy) =\cdots = \lambda_{k + m - 1}(\bsy)$), 
then there exist $m$ FE eigenvalues, 
$\lambda_{k, h}(\bsy),$ $ \lambda_{k + 1, h}(\bsy), \ldots, \lambda_{k + m - 1, h}(\bsy)$, 
that converge to $\lambda_k(\bsy)$, but are not necessarily equal.
As such, we also define $E_h(\lambda_k(\bsy))$ to be the direct sum of all the eigenspaces 
$E(\lambda_{\ell, h}(\bsy))$ such that $\lambda_{\ell, h}(\bsy) \to \lambda_k(\bsy)$.
Finally, we define 
$\Ehat_h(\lambda_k(\bsy) \coloneqq \{ v \in E_h(\lambda_k(\bsy)) : \|v\|_V = 1\}$.

In Assumption~A\ref{asm:coeff} we have only assumed that the physical domain $D$ is
convex and that $a \in W^{1, \infty}(D)$. Hence, piecewise linear FEs are sufficient to
achieve the optimal rates of convergence with respect to $h$ in general.
In particular, in \cite[Thm.~2.6]{GGKSS19} it was shown that the FE
error for the minimal eigenpair can be bounded independently of $\bsy$
with the usual rates in terms of $h$.
Explicitly, if $h > 0$ is sufficiently small, then for all $\bsy \in \Omega$ 
\begin{align}
\label{eq:fe_lam_u}
\nrm{u(\bsy) - u_h(\bsy)}{V} \,&\leq\, C_{u} h, \qquad 
|\lambda(\bsy) - \lambda_h(\bsy)| \,
\leq\, C_{\lambda} h^2,
\end{align}
and for $\calG \in H^{-1 + t}(D)$ with $ t \in [0, 1]$
\begin{equation}
\label{eq:fe_G}
\big|\calG(u(\bsy)) - \calG(u_h(\bsy))\big| \,\leq\,C_{\calG} \,h^{1 + t}\,,
\end{equation}
where $0 < C_{\lambda},\ C_{u},\ C_{\calG} < \infty$ are independent of $\bsy$ and $h$.

In the companion paper \cite{GS20a}, it is shown that for $h$
sufficiently 
small\footnote{The explicit condition is that $h \leq \hsuff$ with
  $\hsuff \coloneqq \sqrt{\rho/(2C_{\lambda})}$.}
the spectral gap of the FE eigenvalue problem \eqref{eq:fe-evp}
satisfies the uniform lower bound
\begin{equation}
\label{eq:gap_h}
\lambda_{2, h}(\bsy) - \lambda_{1, h}(\bsy)
\,\geq\,  \frac{\rho}{2} \,>\, 0,
\end{equation}
and that the eigenvalues and eigenfunctions of both \eqref{eq:var-evp} and
\eqref{eq:fe-evp} satisfy the bounds
\begin{align}
\label{eq:lam_bnd}
\underline{\lambda_k} \,\leq\, \lambda_k(\bsy) \,\leq\, \lambda_{k, h}(\bsy) \,&\leq\,\overline{\lambda_{k}}, \quad \text{and}\\
\label{eq:u_bnd}
\max\big\{\nrm{u_k(\bsy)}{V},\ \nrm{u_{k, h}(\bsy)}{V}\big\}\,&
\,\leq\, \overline{u_{k}},
\end{align}
where $\underline{\lambda_k}, \overline{\lambda_k}, \overline{u_k}$
are also independent of both $\bsy$ and $h$.

Note that the use of piecewise linear FEs is not a restriction
on our MLQMC methods. 
The algorithms presented in Section~\ref{sec:mlqmc-alg}
are very general, and will work with higher order FE methods as well,
without any modification of the overall algorithm structure.

\subsection{Iterative solvers for eigenvalue problems}

The discrete EVP \eqref{eq:fe-evp} from the previous section 
leads to a generalised matrix EVP of the form $A_h \bsu_h = \lambda_h B_h \bsu_h$,
where, in general, the matrices $A_h$ and $B_h$ are large, sparse and
symmetric positive definite.

Since we are only interested in computing a single eigenpair, we will use 
Rayleigh quotient (RQ) iteration to compute it. It is well-known that for 
symmetric matrices RQ iteration converges cubically
for almost all starting vectors, see, e.g., \cite{Parlett80}.

\subsection{Quasi-Monte Carlo integration}
\label{sec:qmc}
A quasi-Monte Carlo (QMC) method is an equal weight quadrature rule
\begin{equation}
\label{eq:qmc}
Q_{s, N} f \,=\, \frac{1}{N} \sum_{k = 0}^{N - 1} f(\bst_k)
\end{equation}
with $N \in \N$ deterministically-chosen quadrature points
$\{\bst_k\}_{k = 0}^{N - 1}$, as opposed to random quadrature
points as in Monte Carlo. The key feature of QMC methods is 
that the points are cleverly constructed to be
well-distributed within high-dimensional domains,
which allows for efficient approximation of high-dimensional integrals such as
\[
\calI_sf \,\coloneqq\, \int_{[- \frac{1}{2}, \frac{1}{2}]^s} f(\bsy) \rd \bsy\,.
\]
There are many different types of QMC point sets, and for further
details we refer the reader  to, e.g., \cite{DKS13}.

In this paper, we use a simple to construct, yet powerful, class of
QMC methods called \emph{randomly shifted rank-1 lattice rules}. 
A randomly shifted lattice rule approximation to $\calI_sf$ 
using $N$ points is given by
\begin{equation}
\label{eq:rqmc}
Q_{s, N}(\bsDelta)f \,\coloneqq\, \frac{1}{N} 
\sum_{k = 0}^{N - 1} f(\{t_{k}  + \bsDelta\} - \tfrac{\boldsymbol{1}}{\boldsymbol{2}})
\end{equation}
where $\bsz \in \N^s$ is the \emph{generating vector} and the points $\bst_k$ are
given by
\[
\bst_k = \bigg\{\frac{k\bsz}{N}\bigg\} \quad \text{for } k = 0, 1, \ldots, N - 1\,,
\]
$\bsDelta \in [0, 1)^s$ is a uniformly distributed \emph{random shift},
and $\{\cdot\}$ denotes the fractional part of each component of a vector.
Note that we have subtracted $1/2$ in each dimension 
to shift the  quadrature points from $[0, 1]^s$
to $[-\frac{1}{2}, \frac{1}{2}]^s$.

Good generating vectors can be constructed in practice using the 
\emph{component-by-component} (CBC) algorithm, 
 or the more efficient \emph{Fast CBC} construction \cite{NC06,NC06np}.
In fact, it can be shown that for functions 
in certain first-order weighted Sobolev spaces such as those introduced in \cite{SW98},
the root-mean-square (RMS) error of a randomly shifted lattice rule using a 
CBC-constructed generating vector achieves almost the optimal rate of $\bigO(N^{-1})$.

To state the CBC error bound, 
we briefly introduce the following specific class of weighted Sobolev 
spaces, which are useful for the analysis of lattice rules. Given a collection of  \emph{weights}
$\bsgamma \coloneqq \{\gamma_\setu > 0 : \setu \subseteq \{1, 2, \ldots, s\}\}$,
which represent the importance of different subsets of variables,
let $\Ws$ be the $s$-dimensional weighted (unanchored) Sobolev space of functions with 
square-integrable mixed first derivatives, equipped with the norm
\begin{equation}
\label{eq:W-norm}
\nrm{f}{\Ws}^2
\,=\, \sum_{\setu \subseteq \{1:s\}} \frac{1}{\gamma_\setu} \int_{[-\frac{1}{2}, \frac{1}{2}]^{|\setu|}} 
\bigg(\int_{[-\frac{1}{2}, \frac{1}{2}]^{s - |\setu|}} \pd{|\setu|}{}{\bsy_\setu} f(\bsy) \, \rd \bsy_{-\setu}\bigg)^2
\rd \bsy_\setu,
\end{equation}
where we use the notation $\{1:s\} = \{1, 2, \ldots, s\}$, $\bsy_\setu = (y_j)_{j \in \setu}$
and $\bsy_{-\setu} = (y_j)_{j \in \{1:s\} \setminus \setu}$.
Then, for $f \in \Ws$ and $N$ a power of 2, the RMS error of a CBC-constructed
randomly shifted lattice rule approximation satisfies
\begin{equation}
\label{eq:cbc_err}
\sqrt{\bbE_\bsDelta \big[ |\calI_s f - Q_{s, N}f|^2\big]} \,\lesssim\, N^{-1 + \delta}\nrm{f}{\Ws},
\qquad \delta > 0,
\end{equation}
where under certain conditions on the decay of the weights $\bsgamma$
the constant is independent of the dimension.
Note that similar results also hold for general $N$, but with $N$ 
on the RHS of \eqref{eq:cbc_err} replaced
by the Euler Totient function, which counts the number of integers less than and coprime to $N$, 
see, e.g., \cite[Theorem 5.10]{DKS13}.
For more details on the general theory of lattice rules see \cite{DKS13},
and for a theoretical analysis of randomly shifted lattice rules for MLQMC
applied to \eqref{eq:evp} see \cite{GS20a}.

The generating vectors given by the CBC algorithm are extensible in
dimension, however, they are constructed for a fixed value of $N$.
By modifying the error criterion that is minimised in each step of the CBC algorithm, 
one can construct a generating vector that works well for a range of values 
of $N$, where now $N$ is given as some power of a prime base, e.g., $N$ is a power of 2.
The resulting quadrature rule is called an \emph{embedded lattice rule}
and was developed in \cite{CKN06}.
Not only do embedded lattice rules work well for a range of values of $N$,
but the resulting point sets are nested. 
Hence, one can improve the accuracy of a previously computed embedded lattice rule
approximation by simply adding the function evaluations corresponding to the new points
to the sum from the previous approximation. 
As will be clear later, the extensibilty in both $s$ and $N$ of embedded lattice rules 
makes them extremely convenient for use in MLQMC methods in practice.
 
Currently there is not any theory for the error of embedded lattice rules,
however, a series of comprehensive numerical tests conducted in \cite{CKN06} 
show empirically that the optimal rate of $N^{-1}$ is still observed,
and that the worst-case error for an embedded lattice  increases
at most by a factor of 1.6 as compared to the normal CBC algorithm with $N$ fixed.

Finally, instead of using a single random shift, in practice it is
better to average over several randomly shifted approximations
that correspond to a small number of independent random shifts.
The practical benefits are (i) that averaging gives more consistent results,
by reducing the chance of using a single ``bad'' shift,
and (ii) that the sample variance of the shifted approximations 
provides a practical error estimate.
Let $\bsDelta^{(1)}, \bsDelta^{(2)}, \ldots,  \bsDelta^{(R)}$ be $R$ independent 
uniform random shifts, then the average of the QMC approximations
corresponding to the random shifts is
\[
\Qhat_{s, N, R}f \,\coloneqq\, \frac{1}{R} \sum_{r = 1}^R Q_{s, N}(\bsDelta^{(r)})f ,
\]
and the mean-square error of $\Qhat_{s, N, R}f$ can be estimated by the sample variance
\begin{equation}
\label{eq:sample_var}
\Vhat[\Qhat_{s, N, R}] \,\coloneqq\, \frac{1}{R(R - 1)} \sum_{r = 1}^R \big[\Qhat_{s, N, R} f - Q_{s, N}(\bsDelta^{(r)})f\big]^2.
\end{equation}

\subsection{Discrepancy theory}
\label{sec:disc}

Much of the modern theory for QMC rules is based on weighted function spaces
as discussed in Section~\ref{sec:qmc}, however, the traditional analysis of QMC rules
is based on the \emph{discrepancy} of the quadrature points.
Loosely speaking, for a given point set the discrepancy measures the difference between the 
number of points that actually lie within some subset of the unit cube and the number of points 
that are expected to lie in that subset if the point set were perfectly uniformly distributed.
This more geometric notion of the quality of a QMC point set will be
useful later when we analyse 
the use of an eigenvector corresponding to a nearby QMC point as the
starting vector for the RQ iteration.

We now recall some basic notation and definitions from the field of
discrepancy theory for a point set $\calP_N = \{\bst_0, \bst_1, \ldots, \bst_{N -
  1}\} \subset [0, 1]^s$ on the unit cube. Note
that by a simple translation the results from this section are also
applicable on $[-\frac{1}{2}, \frac{1}{2}]^s$.
The axis-parallel box with corners $\bsa, \bsb\in [0, 1]^s$ with $a_j
< b_j$ is denoted by $[\bsa, \bsb) \coloneqq [a_1, b_1) \times [a_2, b_2) \times
\cdots \times [a_s, b_s)$. The number of points from 
$\calP_N$ that lie in $[\bsa, \bsb)$ is denoted by
$|\{\calP_N \cap [\bsa, \bsb)\}|$ and the Lebesgue 
measure on $[0, 1]^s$  by $\calL_s$. 

\begin{definition}
\label{def:D*}
The \emph{star discrepancy} of a point set $\calP_N$ is defined by
\begin{equation}
\label{eq:D*}
D_N^*(\calP_N) \,\coloneqq\, \sup_{\bsb \in [0, 1]^s} 
\bigg| \frac{|\{\calP_N \cap [\bszero, \bsb)\}|}{N} - \calL_s\big([\bszero, \bsb)\big)\bigg|.
\end{equation}
$\calP_N$ is called a \emph{low discrepancy} point set if there exists $C_{\calP_N} < \infty$,
independent of $s$, such that
\begin{equation}
\label{eq:low-D_N}
D_N^*(\calP_N) \,\leq\, C_{\calP_N} \frac{\log (N)^{s - 1}}{N}.
\end{equation}
\end{definition}
There exist several well-known points sets that have low-discrepancy,
such as Hammersley point sets, see \cite{DP10} for more details.

The connection between star discrepancy and quadrature is given by the Koksma--Hlawka
inequality, which for a function $f$ with bounded Hardy--Krause variation states
that the quadrature error of a QMC approximation \eqref{eq:qmc}
satisfies the bound
\begin{equation}
\label{eq:KH}
\Bigg| \int_{[0, 1]^s} f(\bsy) \, \rd \bsy - Q_{s, N}f\Bigg|
\,\leq\, \Bigg(\sum_{\emptyset \neq \setu \subseteq \{1:s\}} 
\int_{[0, 1]^{|\setu|}} \bigg| \pd{|\setu|}{}{\bsy_\setu}
f(\bsy_\setu; \bsone)\bigg| \, \rd \bsy_\setu \Bigg) D^*_N(\calP_N),
\end{equation}
see, e.g., \cite{DP10}.
Here,  $(\bsy_\setu; \bsone)$ denotes the \emph{anchored} point
with $j$th component $y_j$ if $j \in \setu$ and $1$ otherwise. 
Hence, low-discrepancy point sets lead to QMC approximations
for which the error converges like $\bigO(\log(N)^{s - 1}/N)$.

Lattice rules can also be constructed such that
their discrepancy is $\log(N)^{s}/N$ (see \cite[Corollary
3.52]{DP10}). By the Koksma--Hlawka inequality \eqref{eq:KH}, 
they then admit error bounds similar to
\eqref{eq:cbc_err}, but with an extra $\log(N)^{s}$ factor.
By instead considering the weighted discrepancy, one can construct lattice rules
that have a weighted discrepancy (and similarly error bounds) without this $\log$ factor,
see \cite{Joe06}.

Finally, we also define the extreme discrepancy of a point set,
which removes the restriction that the boxes are anchored to the origin.
\begin{definition}
\label{def:D_ext}
The \emph{extreme discrepancy} of a point set $\calP_N$ is defined by
\begin{equation}
\label{eq:D_extreme}
\widehat{D}_N(\calP_N) \,\coloneqq\, \sup_{[\bsa, \bsb) \subset [0, 1]^s} 
\bigg| \frac{|\{\calP_N \cap [\bsa, \bsb)\}|}{N} - \calL_s\big([\bsa, \bsb)\big)\bigg|.
\end{equation}
\end{definition}

\section{Two-grid-truncation methods for stochastic EVPs}
\label{sec:2grid}
Two-grid FE discretisation methods for EVPs were
first introduced in \cite{XuZhou99} and later refined independently in 
\cite{HuCheng11, HuCheng15_corr} and \cite{YangBi11}.
The idea behind them is simple: 
to combine FE methods with iterative solvers for matrix EVPs.
Letting $H > h > 0 $ be the meshwidths of a coarse and a fine FE mesh, $\mathscr{T}_H$ and
$\mathscr{T}_h$, respectively, one first solves the EVP \eqref{eq:fe-evp}
on the coarse FE space $V_H$ to give $\lambda_H$, $u_H$. 
This coarse eigenpair ($\lambda_H$, $u_H$) is then used as the starting guess
for an iterative eigensolver for the EVP on the fine FE space $V_h$.
Since FE methods for PDE EVPs and iterative methods for matrix 
EVPs both converge very fast, $H$ and $h$ can be chosen such that
a single linear solve is all that is required to obtain the same order of accuracy as 
can be expected from solving the original FE EVP on the fine mesh. 
This strategy can be adapted to a full multigrid method
for EVPs as in, e.g., \cite{Xie14}.
However, it was shown in \cite{HuCheng11,HuCheng15_corr,YangBi11}
that the maximal ratio $H/h$ between the coarse and fine meshwidth in a two grid method 
is so large that, in general, two grids are sufficient.

Here, we present a new algorithm that extends the two-grid method to
stochastic (or parametric) EVPs, by also using a reduced  (i.e., cheaper and less accurate) truncation 
of the parameter space when solving the parametric EVP on the initial coarse mesh.
Since our new algorithm combines this truncation and FE approximations,
we first introduce some notation. For some $h > 0$ and $s \in \N$, the
FE EVP that approximates the truncated problem \eqref{eq:trunc-evp} is: 
Find  $\lambda_{h, s}(\bsy) \in \R$ and $u_{h, s} (\bsy) \in V_h$ such that
\begin{align}
\label{eq:trunc-fe-evp}
\calA_s(\bsy; u_{h, s}(\bsy), v_h) \,&=\, \lambda_{h, s}(\bsy) \calM(u_{h, s}(\bsy), v_h)
\quad \text{for all } v_h \in V_h\,,\\
\nrm{u_{h, s}(\bsy)}{\calM} \,&=\, 1\,.
\nonumber
\end{align}
We also define the solution operator $T_{h, s} = T_{h, s} (\bsy) : V^*
\to V_h$ for \eqref{eq:trunc-fe-evp}, which for $f \in V^*$ satisfies
\begin{equation}
\label{eq:T_h,s}
\calA_s(\bsy; T_{h, s}  f, v_h) \,=\, \calM(f, v_h) \quad \text{for all } v_h \in V_h,
\end{equation}
and the $\calA_s(\bsy)$-orthogonal projection operator 
$P_{h, s} = P_{h, s}(\bsy): V \to V_h$, which for $u \in V$ satisfies
\begin{equation}
\label{eq:P_h,s}
\calA_s(\bsy; u - P_{h, s} u, v_h) \,=\, 0 \quad \text{for all } v_h \in V_h.
\end{equation}
Although both operators depend on $\bsy$ we will not specify this dependence.

In Algorithm~\ref{alg:2grid} below we detail our new two-grid and truncation method for
parametric EVPs.
The algorithm is based on the accelerated version of the two-grid
algorithm (see \cite{HuCheng11,HuCheng15_corr} and also
\cite{YangBi11}), which uses the shifted-inverse power method for the update step. 
In addition, we add a normalisation step
so that $\nrm{u^h(\bsy)}{\calM} = 1$, which simplifies the RQ update
but does not affect the theoretical results.
As in the papers above, we will consistently use the notation that
two-grid approximations use superscripts whereas ordinary approximations (i.e., eigenpairs of
truncated/FE problems) use subscripts. 
To perform step 2 in practice, one must interpolate the $u_{H, S}(\bsy)$
at the nodes of the fine mesh $\calT_{h}$ to obtain the corresponding start vector.
Since we use piecewise linear FE methods, linear interpolation is sufficient.

\begin{algorithm}[!h]
\caption{Two-grid-truncation method for parametric EVPs}
\label{alg:2grid}
Given $H > h > 0$, $0 < S < s$  and $\bsy \in \Omega$:
\begin{algorithmic}[1]
\State Find $\lambda_{H, S}(\bsy) \in \R$ and $u_{H, S}(\bsy) \in V_H$ such that
\begin{align*}
\calA_S(\bsy; u_{H, S}(\bsy), v_H) \,&=\, \lambda_{H, S}(\bsy) \calM(u_{H, S}(\bsy), v_H)\,
\quad \text{for all } v_H \in V_H,\\
\nrm{u_{H, S}(\bsy)}{\calM} \,&=\, 1\,.
\end{align*}
\State Find $u^{h, s} \in V_h$ such that
\begin{equation}
\label{eq:2grid_source}
\calA_s(\bsy; u^{h, s}(\bsy), v_h) - \lambda_{H, S}(\bsy) \calM(u^{h, s}(\bsy), v_h) 
\,=\, \calM(u_{H, S}(\bsy), v_h)
\quad \text{for all } v_h \in V_h. 
\end{equation}
\State $u^{h, s}(\bsy) \leftarrow u^{h, s}(\bsy)/ \nrm{u^{h, s}(\bsy)}{\calM}$
\Comment normalise the eigenfunction approximation
\State 
\begin{equation}
\label{eq:RQ_update}
\lambda^{h, s}(\bsy) \,=\, \calA_s(\bsy; u^{h, s}(\bsy), u^{h, s}(\bsy))\,.
\end{equation}
\end{algorithmic}
\end{algorithm}

The following lemmas will help us to extend the error analysis of two-grid methods to include a
component that corresponds to truncating the parameter dimension.

\begin{lemma}
\label{lem:RQ-diff}
Let $\calB,\ \widetilde{\calB} : V \times V \to \R$, be two bounded, coercive, symmetric bilinear forms.
Suppose that $(\lambda, u)$ is an eigenpair of
\[
\calB(u, v) \,=\, \lambda \calM(u, v) \quad \text{for all } v \in V\,,
\]
and let $w \in V$. Then
\begin{equation}
\label{eq:RQ-diff}
\frac{\widetilde{\calB}(w, w)}{\calM(w, w)} - \lambda \,=\, 
\frac{\|u - w\|^2_{\widetilde{\calB}}}{\|w\|_\calM^2}
-\lambda \frac{\|u - w\|_\calM^2}{\|w\|_\calM^2}
+ \frac{1}{\|w\|_\calM^2}\big( \calB(u, u - 2w) - \widetilde{\calB}(u, u - 2w)\big)\,.
\end{equation}
\end{lemma}

\begin{proof}
Expanding,
then using the fact that $(\lambda, u)$ is an eigenpair gives
\begin{align*}
\|u - w\|_{\widetilde{\calB}}^2 - &\lambda \|u - w\|_\calM^2 \\
=\,
&\widetilde{\calB}(u, u) + \widetilde{\calB}(w, w) - 2\widetilde{\calB}(u, w)
-\lambda \calM(u, u) - \lambda\calM(w, w) + 2\lambda\calM(u, w)\\
=\, & \widetilde{\calB}(u, u) + \widetilde{\calB}(w, w) - 2\widetilde{\calB}(u, w)
-\calB(u, u) - \lambda\calM(w, w) + 2\calB(u, w)\\
=\, &  \widetilde{\calB}(w, w) - \lambda\calM(w, w) 
- \calB(u, u - 2w) + \widetilde{\calB}(u, u - 2w)\,.
\end{align*}
Dividing by $\|w\|_\calM^2$ and rearranging leads to the desired result.
\end{proof}

\begin{lemma}
Let Assumption~A\ref{asm:coeff} hold, then
\begin{equation}
\label{eq:T-T_hs}
\|T - T_{h, s} \| \,\leq\, C_T (s^{-1/p + 1} + h),
\end{equation}
where $C_T$ is independent of $\bsy$, $s$ and $h$.
\end{lemma}
\begin{proof}
The differential operator $A(y)v = -\nabla \cdot (a(\bsy) \nabla v) + b(\bsy) v$
from the EVP \eqref{eq:evp} fits into the general framework of \cite{DKLeGNS14}.
Defining $T_s \coloneqq T(\bsy_s)$ to be the solution operator 
for the truncated EVP \eqref{eq:trunc-evp}, it follows from the triangle inequality that
\[
\|T - T_{h, s} \| \,\leq\, \|T - T_s\| + \|T_s - T_{h, s}\| \,\leq\, C_1 s^{-1/p + 1} + C_2 h,
\]
where we have used \cite[Theorem 2.6 and eq. (2.17)]{DKLeGNS14} in the
last step.
\end{proof}

The error of the outputs of Algorithm~\ref{alg:2grid} are given in the theorem below.
The proof follows a similar proof technique as used in \cite{YangBi11}, and
also relies on an abstract approximation result for operators from that paper.
Note that the FE component of the error is the same as the results in
\cite{HuCheng11,HuCheng15_corr,YangBi11}, but here we have extra terms
corresponding to the truncation error.
The proof is deferred to the appendix.

\begin{theorem}
\label{thm:2grid}
Suppose that Assumption~A\ref{asm:coeff} holds, let $S \in \N$ be sufficiently large and
let $H > 0$ be sufficiently small. Then, for $ s > S$ and $0 < h < H$, 
\begin{align}
\label{eq:tg_u_err}
\nrm{u(\bsy) - u^{h, s}(\bsy)}{V} \,&\lesssim\, H^4 + h + S^{-2(1/p - 1)} 
+ s^{-(1/p - 1)} + H^2S^{-(1/p - 1)}\,, 
\quad \text{and}\\
\label{eq:tg_lam_err}
|\lambda(\bsy) - \lambda^{h, s}(\bsy)| \,& \lesssim\, H^8 + h^2 + S^{-4(1/p - 1)} + s^{-(1/p - 1)}
+ H^4 S^{-2(1/p - 1)},
\end{align}
where both constants are independent of $s, S, h, H$ and $\bsy$.
\end{theorem}

It follows that in our two-grid-truncation method,
to maintain the optimal order $h$ convergence for the eigenfunction we should
take $H \eqsim h^{1/4}$, $s \eqsim h^{-p/(1 - p)}$ and $S \eqsim s^{1/2} $,
whereas for the eigenvalue error we should take a higher truncation dimension, namely
$s \eqsim h^{-2p/(1 - p)}$ and $S \eqsim s^{1/4}$.
The difference in conditions comes from the fact that for EVPs, the truncation error 
for the eigenvalue and eigenfunction are of the same order,
whereas the FE error for the eigenvalue is double the order of the eigenfunction FE error.
It is similar to how 
a higher precision numerical quadrature rule should be used to compute the elements
of the stiffness matrix for eigenvalue approximation, see, e.g., \cite{Banerjee92}.

\section{MLQMC algorithms for random eigenvalue problems}
\label{sec:mlqmc-alg}

In this section, we present two MLQMC algorithms for approximating
the expectation of a random eigenvalue. First, we briefly give a 
straightforward MLQMC algorithm,
for which a rigorous theoretical analysis of the error was presented in \cite{GS20a}.
After analysing the cost of this algorithm we then present a second, 
more efficient MLQMC algorithm,
where we focus on reducing the overall cost by reducing the cost of
evaluating each sample.

\subsection{A basic MLQMC algorithm for eigenvalue problems}

The starting point of our basic MLQMC algorithm is the telescoping sum \eqref{eq:tele_sum},
along with a collection of 
$L + 1$ FE meshes corresponding to meshwidths, $h_0 > h_1 > \cdots > h_L > 0$,
and $L + 1$ truncation dimensions, 
$0 < s_0 \leq s_1 \leq \cdots \leq s_L < \infty$.
Recall that we denote the eigenvalue approximation on level $\ell$ by 
$\lambda_\ell \coloneqq \lambda_{h_\ell, s_\ell}$ with $\lambda_{-1} \equiv 0$.
The expectation on each level $\ell$ in the sum \eqref{eq:tele_sum} can be approximated
by a QMC rule using $N_\ell$ points, which we denote by
$Q_\ell \coloneqq Q_{s_\ell, N_\ell}$ as in \eqref{eq:rqmc}, so that our MLQMC approximation
of $\bbE[\lambda]$ is
\begin{equation}
\label{eq:mlqmc0}
\Qml_L(\bsDelta) \lambda
\,\coloneqq\, 
\sum_{\ell = 0}^L Q_\ell(\bsDelta_\ell)\big(\lambda_\ell - \lambda_{\ell - 1}\big).
\end{equation}
Here, each $\bsDelta_\ell \in [0, 1)^{s_\ell}$ is an independent random shift, so that
the QMC approximations on different levels are independent. To simplify the notation,
we also concatenate the $L+1$ shifts into a single random shift 
$\bsDelta = (\bsDelta_0; \bsDelta_1; \ldots; \bsDelta_L)$ 
(where ``;''  denotes concatenation of column vectors).
For a linear functional $\calG \in V^*$, the MLQMC approximation to
$\bbE_\bsy[\calG(u)]$ can be defined analogously.

As described in Section~\ref{sec:qmc}, in practice it is beneficial to
use $R$ independent random shifts $\bsDelta^{(1)},$ $\bsDelta^{(2)},
\ldots, \bsDelta^{(R)}$. Then the shift-averaged MLQMC approximation is
\begin{equation}
\label{eq:mlqmc_R}
\Qhatml_{L, R}\lambda \,\coloneqq\, \sum_{\ell = 0}^L \frac{1}{R} \sum_{r = 1}^R
Q_\ell(\bsDelta_\ell^{(r)})\big(\lambda_\ell - \lambda_{\ell - 1}\big)\,.
\end{equation}
In this case, the variance on each level can be estimated by 
the sample variance as given in \eqref{eq:sample_var} and denoted by
$V_\ell$. Due to the independence of the QMC approximations 
across the levels, the total variance of the MLQMC estimator is
\begin{equation}
\label{eq:ml-var}
\Vhat\big[\Qhatml_{L, R}\lambda\big]\,=\, \sum_{\ell = 0}^L V_\ell.
\end{equation}

\subsection{Cost \& error analysis}
\label{sec:cost-err}
The cost of the MLQMC estimator \eqref{eq:mlqmc_R} for the expected value of $\lambda$
is given by
\[
\cost(\Qhatml_{L, R}\lambda) \,=\, R \sum_{\ell = 0}^L N_\ell \; \cost(\lambda_\ell - \lambda_{\ell - 1})\,,
\]
where $\cost(\lambda_\ell - \lambda_{\ell - 1})$ denotes the cost of evaluating the difference
at a single parameter value. Since $\cost(\lambda_\ell - \lambda_{\ell - 1}) \leq 2 \cost(\lambda_\ell)$, 
the cost of evaluating $\lambda_\ell$ at a single parameter value, 
\[
\cost(\Qhatml_{L, R}\lambda) \,\lesssim\, R \sum_{\ell = 0}^L N_\ell \cost(\lambda_\ell)\,.
\]

The cost of evaluating the eigenvalue approximation $\lambda_\ell$ consists of two parts:
\[
\cost(\lambda_\ell) \,=\, \Csetup_\ell + \Csolve_\ell\,,
\]
where $\Csetup_\ell$ denotes the setup cost of constructing the stiffness and mass matrices, 
and $\Csolve$ denotes the cost of solving the eigenvalue problem. Since the coefficient
$c$ is independent of $\bsy$ so too is the mass matrix, and as such we
only compute it once per level. Thus, $\Csetup_\ell$ is dominated by constructing the stiffness matrix 
for each quadrature point.

Constructing the stiffness matrix  at each parameter value involves
evaluating the coefficients, which are $s_\ell$-dimensional sums, 
at the quadrature points for each element in the mesh.
Under the assumption \eqref{eq:fe_dof} on the number of FE degrees of freedom, the
number of elements in the mesh is also $\bigO(h^{-d})$, which implies the setup cost is
\[
\Csetup_\ell \,\lesssim\, s_\ell h_\ell^{-d}\,.
\]

At each each step of an iterative eigensolver a linear system
must be solved, and this forms the dominant component of the cost for that step.
Essentially, the cost of each eigenproblem solve is of the order of a source problem solve
(on the mesh $\mathscr{T}_{h_\ell}$) multiplied by the number of iterations.
As in the case of the source problem (see e.g., \cite{KSS15,KSSSU17}), we assume that the 
linear systems occurring in each iteration of the eigensolver can
be solved in $\bigO(h^{-\gamma})$ operations, with $d < \gamma < d + 1$.
Assuming that the number of iterations required is independent of $\bsy$,
the cost of each eigensolve is then
\begin{equation}
\label{eq:solver_cost}
\Csolve \,\lesssim\, h_\ell^{-\gamma}.
\end{equation}
We discuss how to bound the number of iterations of the eigensolver in Section~\ref{sec:nearby_qmc}.

It then follows that the cost of evaluating $\lambda_\ell$ at a single
parameter value satisfies
$\cost(\lambda_\ell) \lesssim s_\ell h_\ell^{-d} + h_\ell^{-\gamma}$, and 
hence the total cost of the MLQMC estimator \eqref{eq:mlqmc_R}
satisfies
\begin{equation}
\label{eq:cost_lambda}
\cost(\Qhatml_{L, R}\lambda) \,\lesssim\, R\sum_{\ell = 0}^L N_\ell (s_\ell h_\ell^{-d} + h_\ell^{-\gamma})\,.
\end{equation}

Since the eigenfunction approximation $u_\ell$ is computed at the same time as $\lambda_\ell$,
and we assume that the cost of applying a linear functional $\calG$ is constant, the cost of the MLQMC estimator
$\Qhatml_{L, R} \calG(u)$ is of the same order as the cost of the
eigenvalue estimator in \eqref{eq:cost_lambda}.

The error of the approximation \eqref{eq:mlqmc0} is analysed rigourously in \cite{GS20a},
and so here we only give a brief summary of one of the key results.
First, suppose that Assumption~A\ref{asm:coeff} holds with $0 < p < q < 1$,
and let each $Q_\ell$ use a generating vector given by the CBC construction.
Next, choose  $h_\ell \eqsim 2^{-\ell}$ and $s_\ell = s_L \eqsim h_L^{2p/(2 - p)}$.
Then, it was shown in \cite[Corollary~3.1]{GS20a} that for $0 \leq \varepsilon < \exp(-1) $, 
we can choose $L$ and the number of points on each level,
$N_\ell$, such that the mean-square error of the estimator \eqref{eq:mlqmc0} is bounded by
\[
\bbE_\bsDelta\big[|\bbE_\bsy[\lambda] - Q^\mathrm{ML}_L(\bsDelta) \lambda |^2\big]
\,\leq\, \varepsilon^2,
\]
and for $\delta > 0$ the cost is bounded by
\begin{equation}
\label{eq:cost_bound}
\cost\big(Q^{\mathrm{ML}}_{L}(\lambda)\big) \,\lesssim\,
\begin{cases}
\varepsilon^{-2/\eta - p/(2 - p)} 
& \text{if } \eta < 4/d,\\
\varepsilon^{-2/\eta - p/(2 - p)} \log_2(\varepsilon^{-1})^{1 + 1/\eta}
&\text{if } \eta = 4/d,\\
\varepsilon^{-d/2 - p/(2 - p)} 
& \text{if } \eta > 4/d,
\end{cases}
\end{equation}
where $\eta$ is the convergence rate of the variance of
$V_\ell$ with respect to $N_\ell$, which by \cite[Theorem~5.3]{GS20a} is given by\vspace{-1ex}
\[
\qquad \eta \,=\, \begin{cases}
2 - \delta & \text{if } q \in (0, \tfrac{2}{3}]\\[1mm]
\displaystyle \frac{2}{q} - 1 & \text{if } q \in (\tfrac{2}{3}, 1).
\end{cases}
\]

In practice, we typically set $h_\ell \eqsim 2^{-\ell}$, $s_\ell \eqsim 2^\ell$ and use the adaptive algorithm 
from \cite{GilesWater09} to choose $\{N_\ell\}$ and $L$.
Although this is a greedy algorithm, it was shown in \cite[Section 3.3]{KSSSU17} 
that the resulting choice of $N_{\ell}$ leads to the same asymptotic order for the overall cost
as the choice of $N_{\ell}$ in the theoretical complexity estimate from \cite[Corollary~3.1]{GS20a}.

\subsection{An efficient MLQMC method with reduced cost per sample}
\label{sec:tg-mlqmc}

To reduce the cost of computing each sample in
the MLQMC approximation \eqref{eq:mlqmc0} in practice, we employ the following two strategies
for each evaluation of the difference $\lambda_\ell - \lambda_{\ell - 1}$
on a given level: 1) we use the two-grid-truncation method (cf.~Algorithm~\ref{alg:2grid}) to evaluate 
the eigenpairs in the difference; and 
2) we use the eigenvector from a nearby quadrature point as the starting vector for the 
eigensolve on the coarse mesh.

\paragraph{Two-grid-truncation methods}
Our strategy for how to use the two-grid-truncation method from Section~\ref{sec:2grid}
for a given sample $\bsy$ is as follows.
First, we solve the EVP \eqref{eq:2grid_source}
corresponding to a coarse discretisation, 
with meshwidth and truncation dimension given by
\[
H_\ell = \min\big(h_\ell^{1/4}, h_0\big)
\quad \text{and} \quad
S_\ell = \max\big( \big\lceil s_\ell^{1/2} \big\rceil, s_0\big),
\]
to get the \emph{coarse} eigenpair $(\lambda_{H_\ell, S_\ell}(\bsy), u_{H_\ell, S_\ell}(\bsy))$.
Then,  we let 
$u^\ell(\bsy) \coloneqq u^{h_\ell, s_\ell}(\bsy)  \in V_\ell$ 
be the solution to the following 
source problem
\begin{equation}
\label{eq:2grid_src_ell}
\calA_{s_\ell}(\bsy; u^\ell(\bsy), v) - 
\lambda_{H_\ell, s_\ell}(\bsy)\calM(u^\ell(\bsy), v)
\,=\, \calM(u_{H_\ell, s_\ell}(\bsy), v)
\Forall v \in V_\ell\,,
\end{equation}
and define the eigenvalue approximations for $\ell = 1, 2, \ldots, L$ by the Rayleigh quotient
\begin{equation}
\label{eq:2grid_update_ell}
\lambda^{\ell}(\bsy) \,\coloneqq\, \lambda^{h_\ell}_{s_\ell}(\bsy) 
\,\coloneqq\, \frac{\calA_{s_\ell}(\bsy; u^\ell(\bsy), u^\ell(\bsy))}{\calM(u^\ell(\bsy), u^\ell(\bsy))}\,.
\end{equation}

The eigenpair on level $\ell - 1$ for the same sample is computed in
the same way and we set $\lambda^{-1} (\bsy) = 0$ and $\lambda^{0}(\bsy)  =  \lambda_{h_0, s_0}(\bsy) $.

In this way, the MLQMC approximation with two-grid update, and $R$ random shifts, is given by
\begin{equation}
\label{eq:tg-mlqmc}
\Qhattg_{L, R} \lambda
\,\coloneqq\, 
\frac{1}{R} \sum_{r = 1}^R \sum_{\ell = 0}^L 
Q_\ell(\bsDelta_\ell^{(r)})\big(\lambda^\ell - \lambda^{\ell - 1}\big).
\end{equation}

Note that for a given sample on level $\ell$ we use $(\lambda_{H_\ell, S_\ell}(\bsy), u_{H_\ell, S_\ell}
(\bsy))$ to compute $\lambda^{\ell - 1}$ as well. Technically, this violates the telescoping property,
since $\lambda^{\ell - 1}$ from the previous level ($\ell - 1$) will use 
$(\lambda_{H_{\ell - 1}, S_{\ell - 1}}(\bsy), u_{H_{\ell - 1}, S_{\ell - 1}} (\bsy))$, but in practice this 
difference is negligible and does not justify an extra coarse solve.
Furthermore, since the two-grid method allows for such a large difference in parameters of the
coarse grid and the fine grid ($H \eqsim h^{1/4}$ and $S \eqsim s^{1/2}$), 
often we will have the case where 
$H_{\ell - 1} = H_\ell = h_0$ and $S_{\ell - 1} = S_\ell = s_0$. So that 
reusing $(\lambda_{H_\ell, S_\ell}(\bsy), u_{H_\ell, S_\ell}(\bsy))$ to compute
$\lambda^{\ell - 1}$ in the difference on level $\ell$ does not violate the telescoping property.
As an example, if we take $h_0 = 1/8$, then we can use $H_\ell = h_0$ as the coarse
meshwidth for all levels $\ell$ up to $h_\ell \leq 2^{-12} = 1/4096$.
In all of our numerical results, $1/4096$ was below the finest grid size $h_L$ required.

Using the two-grid method still involves solving a source problem on
the fine mesh, so that the cost of a two-grid solve is of the same order as 
$\Csolve$ but with an improved constant. The reduction in cost is proportional to 
the number of RQ iterations that are required to solve the
eigenproblem on the fine mesh without two-grid acceleration, 
so that the highest gains will be achieved for problems where the RQ iteration converges slowly.

\paragraph{Reusing samples from nearby QMC points}

Now, for each sample we must still solve the EVP \eqref{eq:2grid_src_ell}
corresponding to a coarse mesh and a reduced truncation dimension, which 
we do using the RQ algorithm (see, e.g., \cite{Parlett80}).
To reduce the number of RQ iterations to compute this coarse eigenpair 
at some QMC point $\bst_k$, we use 
the eigenvector from a nearby QMC point (say $\bst'$) as the starting vector:
$v_0 = u_{H_\ell, S_\ell}(\bst')$. For the initial shift in the RQ algorithm we use
the Rayleigh quotient of this nearby vector with respect to the bilinear form
at the \emph{current} QMC point: $\sigma_0 = \calA_{S_\ell}(\bst_k; v_0, v_0)$.
In practice, we have found that a good choice of the nearby QMC point is simply the previous
point $\bst_{k - 1}$.

Explicit details on how these two strategies are implemented to construct 
the estimator \eqref{eq:tg-mlqmc} in practice are given in Algorithm~\ref{alg:MLQMC}.
First we introduce some notation to simplify the presentation.
Denote the $k$th randomly shifted rank-1 lattice point on level $\ell$ by
\begin{equation}
\label{eq:t_kl}
\bst_{\ell, k} \,\coloneqq\, \left\{\frac{k\bsz_\ell}{N_\ell} + \bsDelta_\ell\right\},
\end{equation}
where $\bsz_\ell$ is an $s_\ell$-dimensional generating vector and 
$\bsDelta_\ell \sim \Uni[0, 1)^{s_\ell}$.

\begin{algorithm}[t]
\caption{Two-grid MLQMC for eigenvalue problems}
\label{alg:MLQMC}
Given $v_0$, $L$, $R$, $\{s_\ell\}_{\ell = 0}^L$, $\{h_\ell\}_{\ell = 0}^L$ 
and $\{N_\ell\}_{\ell = 0}^L$:
\begin{algorithmic}[1]
\For{$\ell = 0, 1, 2, \ldots, L$}
	\State $H_\ell \leftarrow \min (h_\ell^{1/4}, h_0)$ and 
	$S_\ell \leftarrow \max (s_\ell^{1/2}, s_0)$
	\For{$r = 1, 2, \ldots, R$}
		\State generate $\bsDelta_\ell \sim \Uni[0, 1)^{s_\ell}$
	   \For{$k = 0, 1, \ldots, N_\ell$}
	      \State generate $\bst_{\ell, k}$ using the shift $\bsDelta_\ell$ as in \eqref{eq:t_kl}
	      \Comment shifted QMC point
	      \State compute $(\lambda_{H_\ell, S_\ell}(\bst_{\ell,
                k}), u_{H_\ell, S_\ell}(\bst_{\ell, k}))$ using $v_0$ as start value
  	      \State $v_0 \leftarrow u_{H_\ell, S_\ell}(\bst_{\ell, k})$
	      \Comment update starting value
	      \If{$\ell > 0$}
	      \State solve the source problem \eqref{eq:2grid_src_ell} for $u^\ell(\bst_{\ell, k}) \in V_\ell$
        \State set  $\lambda_{H_{\ell-1},S_{\ell-1}} \leftarrow  \lambda_{H_\ell,S_\ell}$
	         and $u_{H_{\ell-1},S_{\ell-1}} \leftarrow  u_{H_{\ell},S_{\ell}}$
	      \State solve the source problem \eqref{eq:2grid_src_ell} for 
	      $u^{\ell - 1}(\bst_{\ell, k}) \in V_{\ell - 1}$
	      \State $\displaystyle \lambda_\ell(\bst_{\ell, k}) \leftarrow 
	                  \frac{\calA_{s_\ell}(\bst_{\ell, k}; u^\ell(\bst_{\ell, k}), u^\ell(\bst_{\ell, k}))}
	                  {\calM(u^\ell(\bst^{\ell, k}), u^\ell(\bst_{\ell, k}))}$
			\Comment two-grid updates	               
			\State $\displaystyle \lambda_{\ell - 1}(\bst_{\ell, k}) \leftarrow 
	                  \frac{\calA_{s_{\ell - 1}}(\bst_{\ell, k}; u^{\ell - 1}(\bst_{\ell, k}), u^{\ell - 1}(\bst_{\ell, k}))}{\calM(u^{\ell - 1}(\bst_{\ell, k}), u^{\ell - 1}(\bst_{\ell, k}))}$
			\EndIf
	      \State $Q^{(r)}_\ell\lambda \leftarrow Q^{(r)}_\ell\lambda + (\lambda^\ell(\bst_{\ell, k}) - \lambda^{\ell - 1}(\bst_{\ell - 1, k}))$
	      \Comment update QMC sum
	   \EndFor
	   \State $\displaystyle Q^{(r)}_\ell\lambda \leftarrow \frac{1}{N_\ell} Q^{(r)}_{\ell} \lambda$
	   \State $\Qhat_{\ell, R}\lambda \leftarrow \Qhat_{\ell, R} + Q^{(r)}_\ell \lambda$
   \EndFor
   \State $\displaystyle\Qhat_{\ell, R}\lambda\leftarrow \frac{1}{R}\Qhat_{\ell, R}\lambda$
   \Comment average over shifts
   \State $\displaystyle\Qhatml_{R, L}\lambda \leftarrow \Qhatml_{L, R}\lambda + \Qhat_{\ell, R}\lambda$
   \Comment update ML estimator
\EndFor
\end{algorithmic}
\end{algorithm}

Finally, by relaxing the restriction that approximations on different levels
are independent from one another we can use the same set of random 
shifts for all levels. In this case, the variance decomposition \eqref{eq:ml-var}
becomes an inequality with a factor $L$ in front of the sum.
Following the arguments in \cite[Section~3.1 and Remark 2]{BierChern16}
it can be shown that this does not significantly change the overall complexity,
at worst the cost increases by a factor of $|\log(\epsilon)|$.

Then, if we also use nested QMC rules
we can reuse approximations from lower levels on the higher levels.
In particular, for $\ell \geq 1$ we will have $\bst_{\ell, k} = \bst_{0, k}$ and can set $H_\ell = h_0$ 
so that we can omit 
the coarse eigenvalue solves (steps 6 and 8) in Algorithm~\ref{alg:MLQMC}.
Furthermore, there is no need to calculate $\lambda_{\ell - 1}$, $u_{\ell - 1}$ again either,
so steps 12 and 14 can also be skipped.

In this case, because the optimal choice for the parameters in
the two-grid-truncation methods are $H \eqsim h^{1/4}$ and $S \eqsim s^{1/2}$, 
the range of possible meshwidths and truncation dimensions are restricted. 
With this in mind, we let meshwidth of the 
finest triangulation be denoted by $h$ and let the coarsest possible triangulation have 
$h_0 \leq h^{1/4}$, then we define the maximum number of levels $L$ and the meshwidth 
on each level so that
\[
h \,=\, h_L \,<\, h_{L - 1} \,<\, \cdots \,<\, h_1 \,<\, h_0 \,\leq\, h^{1/4}\,.
\]
For example, if $h = 2^{-8}$, then $h_0 = 2^{-2}$ and
we could take $ L = 6$ with $h_\ell = 2^{-\ell - 2}$. 
Again, this does not affect the asymptotic complexity bounds proved in
\cite{GS20a}.
To overcome this restriction on the coarse and fine meshwidths,
one could instead use a full multigrid method as in \cite{RobNuyVan19}.
Such an extension would be an interesting topic for future work.

Note that for a given problem, 
it may be possible that a meshwidth of $h^{1/4}$ is not sufficiently fine to resolve the coefficients.
For this reason, we only demand that $h_0 \leq h^{1/4}$ and not equality. 
Note also that, asymptotically, the coarsest meshwidth $h_0$
must decrease with the finest meshwidth $h_L$, but only at the rate $h_0 \eqsim h_L^{1/4}$.
Similarly, defining $s_L$ to be the highest truncation dimension, 
the lowest truncation dimension increases like $s_0 \eqsim s_L^{1/2}$.

\subsection{Analysis of using nearby QMC samples}
\label{sec:nearby_qmc}

The argument for why starting from the eigenvector of a nearby QMC point
reduces the number of RQ iterations is very intuitive:
As the number of points $N$ in a QMC rule increases the points necessarily become
closer, and since the eigenvectors are Lipschitz in the parameter 
(see \cite[Proposition~2.3]{GGKSS19}) this implies that the eigenvectors
corresponding to nearby QMC samples become closer as $N$ increases.
Hence the starting guess for the RQ algorithm becomes
closer to the eigenvector that is to be found, and so for a fixed
tolerance the number of RQ iterations also decreases.

In this section we provide some basic analysis to justify our intuition above. 
Throughout it will be convenient to use the more geometric notions from the
classical discrepancy theory of QMC point sets on the unit cube $[0, 1]^s$, 
which were discussed in Section~\ref{sec:disc}.
From the definition of the star discrepancy (see Definition~\ref{def:D*}) 
follows a simple upper bound on how
close nearby points are in a low-discrepancy point set.
The result is given in terms of $\dist(\cdot, \cdot)$, the distance function with respect
to the $\ell^\infty$ norm.

\begin{proposition}
\label{prop:dist_P_N}
Let $\calP_N$ be a low-discrepancy point set for $N > 1$, then
\begin{equation}
\label{eq:dist_P_N}
\max_{\bst \in \calP_N} \dist(\bst, \calP_N \setminus \{\bst\}) \,\leq\,  
3C_{\calP_N}^{1/s} \log (N)^{1 - 1/s} N^{-1/s},
\end{equation}
where $C_{\calP_N}$ is the constant from the discrepancy bound \eqref{eq:low-D_N}
on $\calP_N$.
\end{proposition}

\begin{proof}
Let $\bst \in \calP_N$.
Clearly $\dist(\bst, \calP_N \setminus \{\bst\}) \,\leq\, 1$ holds
trivially because
$\sup_{\bsx, \bsy \in [0, 1]^s} \nrm{\bsx - \bsy}{\ell^\infty} \leq1$.
Hence, we can assume, without loss of generality, that the upper bound in 
\eqref{eq:dist_P_N} satisfies
\begin{equation}
\label{eq:asm_bnd<1}
3C_{\calP_N}^{1/s} \log (N)^{1 - 1/s} N^{-1/s} \,<\, 1,
\end{equation}
which will be satisfied for $N$ sufficiently large.

For any box 
$[\bsa, \bsb) \subset [0, 1]^s$, it follows from the definition of the 
\emph{extreme discrepancy} $\widehat{D}_N$ in
Definition~\ref{def:D_ext} that
\[
\Bigg| \frac{|\{\calP_N \cap [\bsa, \bsb)\}|}{N} - \calL_s\big([\bsa, \bsb)\big) \Bigg|
\,\leq\, 
\sup_{\bsa \le \bsb \in [0, 1]^s}
\bigg| \frac{|\{\calP_N \cap [\bsa, \bsb)\}|}{N} - \calL_s\big([\bsa, \bsb)\big)\bigg|
\,\eqqcolon\, \widehat{D}_N(\calP_N).
\]
By the reverse triangle inequality it then follows that
\begin{equation}
\label{eq:P_N_lower}
|\{\calP_N \cap [\bsa, \bsb)\}| 
\,\geq\, 
N\big(\calL_s\big([\bsa, \bsb)\big) - \widehat{D}_N(\calP_N) \big).
\end{equation}

Now, define $\tau = \big(2/N + \widehat{D}_N(\calP_N)\big)^{1/s}$ and 
consider the box $[\bsa, \bsb)$ given by
\begin{equation}
\label{eq:def_box}
[a_j, b_j) \,=\,
\begin{cases}
[t_j, t_j + \tau) & \text{if } t_j + \tau < 1,\\
[1 - \tau, 1) & \text{otherwise.}
\end{cases}
\end{equation}
From \cite[Proposition 3.14]{DP10},
the extreme discrepancy can be bounded by the star discrepancy:
$\widehat{D}_N(\calP_N) \leq 2^s D_N^*(\calP_N)$, and then due to 
\eqref{eq:low-D_N} and \eqref{eq:asm_bnd<1} we have the upper bound
\begin{equation}
\label{eq:c_upper}
\tau \,\leq\, 
\bigg(\frac{2}{N} + 2^sC_{\calP_N} \frac{\log (N)^{s - 1}}{N}\bigg)^{1/s}
\,\leq\,
3C_{\calP_N}^{1/s} \log (N)^{1 - 1/s} N^{-1/s}
\,<\, 1,
\end{equation}
where we have also used the fact that 
$C_{\calP_N}\log(N)^{s - 1} > 1$ for $N$ sufficiently large and $2 + 2^s \leq 3^s$.
As such, we have $[\bsa, \bsb) \subset [0, 1]^s$, with $\bst \in [\bsa, \bsb)$,
and $\calL_s([\bsa, \bsb)) = \tau^s < 1$.

Applying the lower bound \eqref{eq:P_N_lower} to the box $[\bsa,
\bsb)$ defined in \eqref{eq:def_box} gives
\[
|\{\calP_N \cap [\bsa, \bsb)\}| 
\,\geq\, 
N\big(\tau^s - \widehat{D}_N(\calP_N) \big)
\,=\, N\big(2/N + \widehat{D}_N(\calP_N) - \widehat{D}_N(\calP_N) \big)
\,=\, 2,
\]
which implies that there are at least 2 points in the box $[\bsa, \bsb)$.
By the construction of the box $[\bsa, \bsb)$ it then follows from 
\eqref{eq:c_upper} that there exists
a $\bst' \in \calP_N$ such that $\bst' \neq \bst$ and
\[
\nrm{\bst - \bst'}{\ell^\infty} 
\,\leq\, \tau
\,\leq\, 3C_{\calP_N}^{1/s} \log (N)^{1 - 1/s} N^{-1/s}.\vspace{-2ex}
\]
\end{proof}

Since the eigenvalue and eigenfunction are analytic and thus Lipschitz in $\bsy$, 
we can now bound how close
eigenpairs corresponding to nearby QMC points are, explicit in~$N$.
\begin{proposition}
\label{prop:nearby_eigenpair}
Let $\calP_N$ be a low-discrepancy point set, let $s \in \N$, let $h > 0$ be sufficiently small 
and suppose that Assumption~A\ref{asm:coeff} holds.
Then for any $\bst \in \calP_N$ there exists $\bst \neq \bst' \in \calP_N$ such that the 
eigenvalue and eigenfunction satisfy
\begin{align}
\label{eq:lam_nearby_QMC}
|\lambda_{h, s}(\bst) - \lambda_{h,s}(\bst')| 
\,&\lesssim\, \log (N)^{1 - 1/s} N^{-1/s},
\quad \text{and }\\
\label{eq:u_nearby_QMC}
\nrm{u_{h, s}(\bst) - u_{h, s}(\bst')}{V} \,&\lesssim\, \log (N)^{1 - 1/s} N^{-1/s},
\end{align}
where the constants are independent of $\bst,\, \bst'$, $s$ and $h$.
\end{proposition}

\begin{proof}
We only prove the result for the eigenfunction. The eigenvalue result follows
the same argument.
For $h$ sufficiently small, the eigenfunction $u_h$ is analytic. 
In particular, $u_h$ admits a Taylor series that converges in $V$ for all $\bsy \in \Omega$.
Hence, for any $\bsy, \bsy' \in \Omega$
the zeroth order Taylor expansion of $u_h(\bsy)$ about $\bsy'$ (see \cite{Hoer03}) gives
\[
u_h(\bsy) \,=\, u_h(\bsy') + \sum_{j = 1}^\infty (y_j - y_j') 
\int_0^1 \pdy^j u_h(\tau\bsy) \rd \tau.
\]
Rearranging and taking the $V$-norm, this can be bounded by
\begin{align*}
\nrm{u_h(\bsy) - u_h(\bsy')}{V} \,&\leq\, \nrm{\bsy - \bsy'}{\ell^\infty} \sum_{j = 1}^\infty 
\sup_{\tau \in [0, 1]} \nrm{\pdy^j u_h(\tau\bsy)}{V}
\\
&\leq\, \nrm{\bsy - \bsy'}{\ell^\infty} \sum_{j = 1}^\infty \ubar C_\bsbeta
\max \big(\|a_j\|_{L^\infty}, \|b_j\|_{L^\infty}\big),
\end{align*}
where in the last inequality we have used the upper bound \cite[eq.~(4.4)]{GS20a} on 
the stochastic derivatives of $u_h$, and $C_\bsbeta$ is independent of $h$ and $\bsy$.
From Assumption~A\ref{asm:coeff}.\ref{itm:summable} the sum is finite, and hence
$u_h$ is globally Lipschitz in $\bsy$ with a constant that is independent of $h$. 
Since this bound holds for all $\bsy$, it also holds for all $\bsy$
with $y_j =0$ for $j > s$, and thus clearly $u_{h, s}$ is also
Lipschitz with a constant that is independent of $s$ and $h$.

The Lipschitz continuity of $u_{h, s}$ together with Proposition
\ref{prop:dist_P_N} then imply \eqref{eq:u_nearby_QMC}. Since $C_{\calP_N}^{1/s} \leq \max
( 1, C_{\calP_N})$, the result holds with a constant independent of $s$.
\end{proof}

Suppose now that for $\bst \in \calP_N$ we wish to compute the eigenpair 
$(\lambda_{h, s}(\bst), u_{h, s}(\bst))$
using the RQ algorithm with the initial vector $v_0 = u_{h, s}(\bst')$ and
initial shift $\sigma_0 = \calA_s(\bst; v_0, v_0)$, where $\bst' \in \calP_N$
is the nearby QMC point from Proposition~\ref{prop:nearby_eigenpair}.
Then, there exists an $N$ sufficiently large, such that 
these starting values satisfy
\begin{align*}
\nrm{u_{h, s}(\bst) - v_0}{V} \,<\, 1, 
\qquad \text{and} \qquad
|\lambda_{2, h, s}(\bst) - \sigma_0| \,&>\, \frac{\rho}{2},
\end{align*}
i.e., distance between the initial vector and the eigenvector to be found is less than one, and 
 the initial shift is closer to $\lambda_{h, s}(\bst)$ than to $\lambda_{2, h, s}(\bst)$.
In particular, for any $\bst \in \calP_N$ we can choose the starting values such that this holds.

Since the RQ algorithm converges cubically (see, e.g., \cite{Parlett80})
for all sufficiently close starting vectors,
for a fixed tolerance $\varepsilon > 0$ it follows that the number of 
iterations will be bounded independently of the current QMC point $\bst$,
if the starting vector is sufficiently close.
For $N$ sufficiently large,
Proposition~\ref{prop:dist_P_N} implies that for each QMC point
there is a starting vector
(taken to be the eigenvector corresponding to a nearby QMC point)
that is sufficiently close to the target eigenvector,
with a uniform upper bound on the distance \eqref{eq:dist_P_N}.
This uniform upper bound
implies that for all QMC points the target eigenvector
and the starting vector will be sufficiently close, and 
hence that the number of RQ iterations is bounded independently of the QMC point.
Furthermore, as $N$ increases the starting vector 
becomes closer to the eigenvector to be found due to
\eqref{eq:u_nearby_QMC}, and so the number of iterations
decreases with increasing $N$.

\section{Numerical results}\label{sec:num}

In this section we present numerical results for two different test problems,
which demonstrate the efficiency of MLQMC and also show the computational gains
achieved by our efficient MLQMC algorithm using two-grid methods
and nearby QMC points as described in Section~\ref{sec:tg-mlqmc}.
The superiority of MLQMC for the two test problems is also clearly 
demonstrated by a comparison with single level Monte Carlo (MC), multilevel Monte Carlo (MLMC)
and single level QMC.
All tests were performed on a single node of the computational cluster Katana
at UNSW Sydney. Note also that we use ``e'' notation for powers of 10, e.g., 5e$-3 = 5 \times 10^{-3}$.

The number of quadrature points for all methods ($N$ or $N_\ell$), 
including the MC/MLMC tests,
are chosen to be powers of 2, and for the QMC methods we use a randomly
shifted embedded lattice rule \cite{CKN06} in base 2
given by the generating vector \texttt{lattice-39102-1024-} \texttt{1048576.3600} from 
\cite{KuoLattice} with $R = 8$ random shifts. 
For base-2 embedded lattice rules, the points are enumerated in blocks of powers of
2, where each subsequent block fills in the gaps between the previous points 
and retains a lattice structure, see \cite{CKN06} for further details.
The FE triangulations are uniform, with geometrically decreasing meshwidths
given by $h_\ell = 2^{-(\ell + 3)}$, $\ell \ge 0$. For the two-grid method,
we take as the coarse meshwidth $H_\ell = h_0 = 2^{-3} = 1/8$, which satisfies
$H_\ell \le h_\ell^{1/4}$ for all $\ell \leq 10$. 
Note that none of our tests required a FE mesh as fine as $h = h_{10} = 1/4096$.
To choose $N_\ell$ and $L$, we use the adaptive MLQMC algorithm from
\cite{GilesWater09}, with error tolerances ranging from $\varepsilon  = 0.625, \ldots, 6.1$e$-5$.
For the eigensolver we use the RQ algorithm with an absolute error tolerance of 5e$-8$,
which is below the smallest error tolerance $\varepsilon$ given as input to our 
MLQMC algorithm.

Numerical tests in \cite{GGKSS19} for almost the same EVPs,
show that the error corresponding to dimension truncation with $s =
64$ is less than 1e$-5$. The smallest error tolerance we use below is
bigger than 5e$-5$. Thus, for simplicity we take a single truncation 
dimension $s_\ell = s = 64$ for all $\ell$ below. Consequently, the 
``coarse'' truncation dimension for the two-grid method is then 
$S_\ell = S = s^{1/2} = 8$ for all $\ell$.

\subsection{Problem 1}

First let $D = (0, 1)^2$ and consider the eigenvalue problem
\eqref{eq:evp} with $b \equiv 0$, $c \equiv 1$ and $a$ as in
\eqref{eq:coeff} with $a_0 =1$ or $\pi/\sqrt{2}$
and 
\begin{equation}
a_j(\bsx) \,=\, \frac{1}{j^{\widetilde{p}}} \sin (j \pi x_1) \sin((j +
1)\pi x_2),
\end{equation}
for several different values of the decay parameter $\widetilde{p} > 1$.

Taking the $L^\infty(D)$ norm of the basis functions we get 
$\|a_j\|_{L^\infty} = j^{-\widetilde{p}}$, so that
for all $\widetilde{p}$ the bounds on the coefficient are given by
\[
\amin \,=\, a_0 - \frac{\zeta(\widetilde{p})}{2}
\quad \text{and} \quad
\amax \,=\, a_0 + \frac{\zeta(\widetilde{p})}{2},
\]
where $\zeta$ is the Riemann Zeta function and thus, for
$\widetilde{p} < 2$, a choice of $a_0 = \pi/\sqrt{2}$ ensures 
$\amin > 0$. For $\widetilde{p} \ge 2$ we choose $a_0 = 1$.
Furthermore
\begin{equation}
\nabla a_j(\bsx) \,=\, 
\begin{pmatrix}
\frac{j\pi}{j^{\widetilde{p}}} \cos(j\pi x_1) \sin((j + 1)\pi x_2)
\\[3mm]
\frac{(j + 1)\pi}{j^{\widetilde{p}}} \sin(j\pi x_1) \cos((j + 1)\pi x_2)
\end{pmatrix},
\end{equation}
so that $\|a_j\|_{W^{1, \infty}} = (j + 1)\pi/j^{\widetilde{p}} \leq 2\pi j^{-(\widetilde{p} - 1)}$.
Thus, Assumption~A\ref{asm:coeff} holds for $p > 1/\widetilde{p}$ and 
$q > 1/(\widetilde{p} - 1)$. 

\begin{figure}[!t]
\includegraphics[scale=0.36]{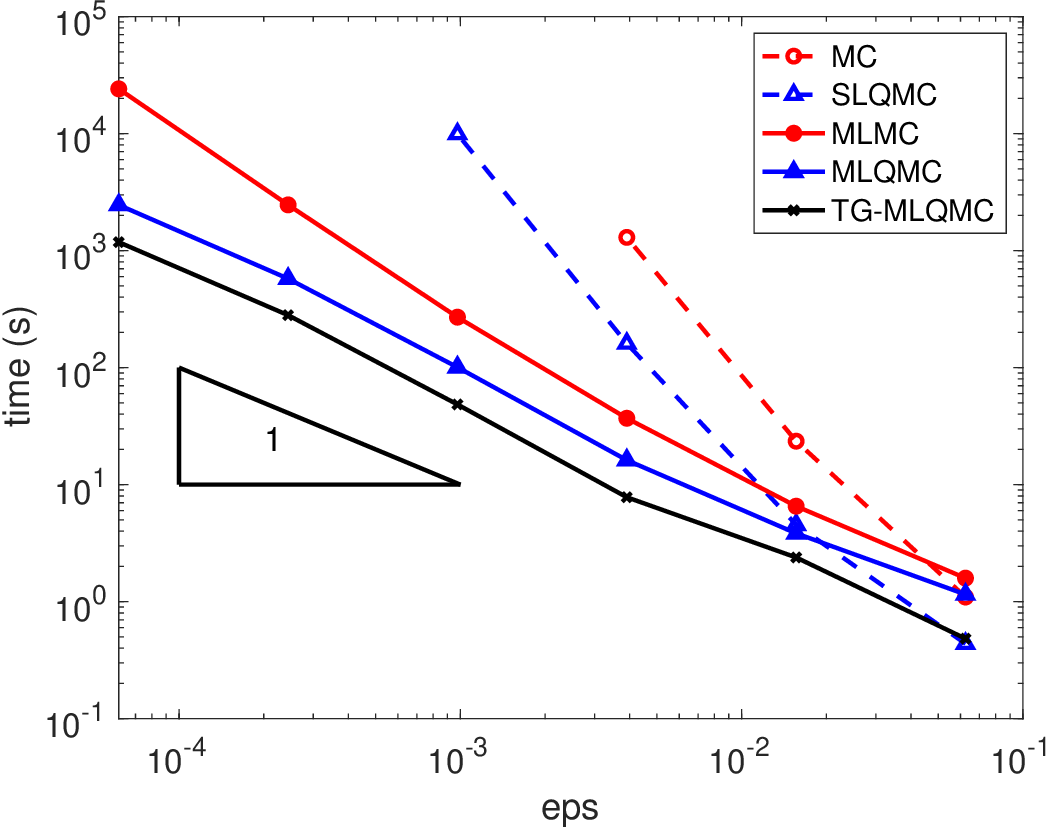}
\hfill
\includegraphics[scale=0.36]{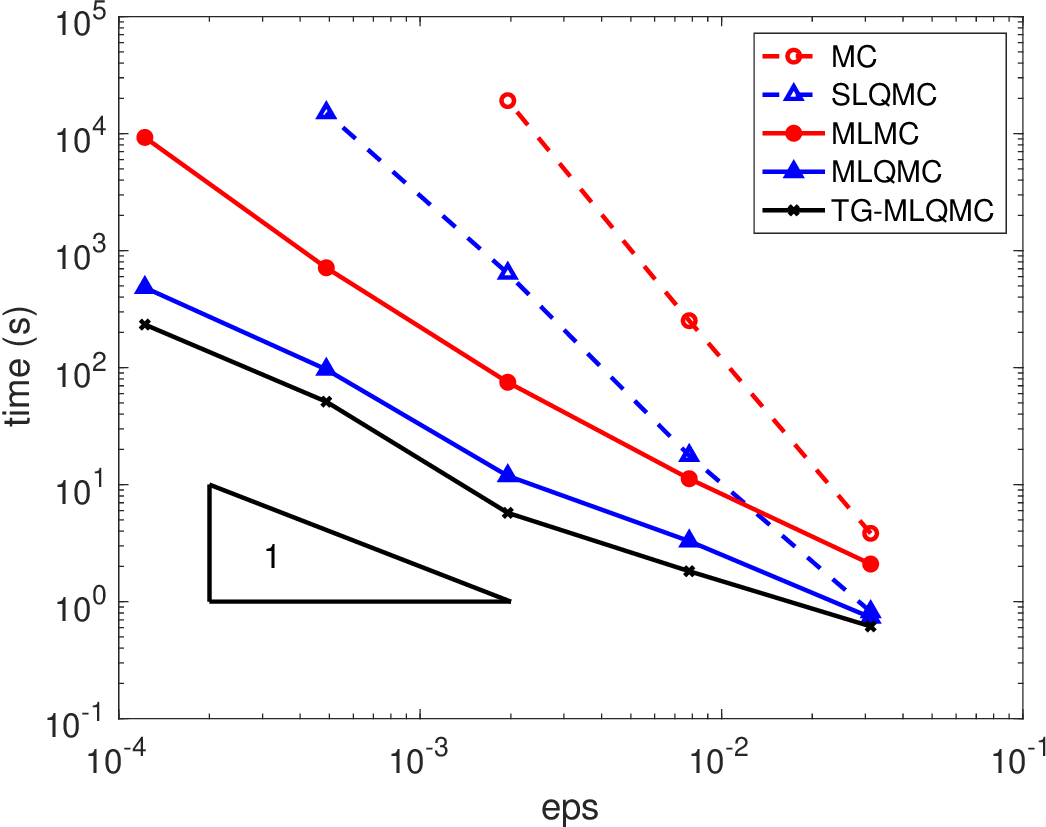}
%
\vspace{-3mm}
\caption{Problem 1: Complexity (measured as time in seconds)
of MC, QMC, MLMC, plain-vanilla MLQMC and the enhanced MLQMC method
using the two-grid method and nearby QMC points
for $\widetilde{p} = 4/3$ (left) and $\widetilde{p} = 2$ (right).
\label{fig:cost1}}
\end{figure}

In Figure~\ref{fig:cost1}, we plot the cost, measured as computational time in seconds,
against the tolerance $\varepsilon$ for our two MLQMC algorithms,
benchmarked against single level MC and QMC, and against MLMC.
To ensure an identical bias error for the single- and multilevel methods, the FE
meshwidth for the single-level methods is taken to be equal to $h_L$, the 
meshwidth on the finest level of the multilevel methods.
The decreasing sequence of tolerances $\varepsilon$ corresponds to a reduction 
in the finest meshwidth $h_L$ by a factor 2 at each step.
The axes are in log-log scale, and as a guide the black triangle in
the bottom left corner of each plot indicates a slope of $-1$. 
As expected, the MLQMC algorithms are clearly superior in all cases,
and for $\widetilde{p} = 2$, the MLMC and single level QMC methods seem to 
converge at the same rate of approximately $-2$.
Also, as we expect the cost of the two MLQMC algorithms grow at the same rate 
of roughly $-1$,
but the enhancements introduced in Algorithm~\ref{alg:MLQMC} yields a
reduction in cost by a (roughly) constant factor of about 2.
Note that for this problem the RQ algorithm requires only 3 iterations 
for almost all cases tested, and so at best we can expect a speedup factor of 3.
In almost all of our numerical tests using the eigenvector of a nearby QMC point as 
the starting vector reduced the number of RQ iterations to 2.
A similar speedup by a factor of 2 was also observed in \cite{RobNuyVan19},
which recycled samples from the multigrid hierarchy within a MLQMC algorithm for
the elliptic  source problem.

From \cite[Corollary~3.1]{GS20a},  for our MLQMC algorithms 
we expect a rate of $-1$ (with a log factor)  when 
$q \leq 2/3$, or equivalently $\widetilde{p} \geq 5/2$.
However, we observe for our MLQMC algorithms are close to $-1$, regardless
of the decay $\widetilde{p}$. A possible explanation of this is that we use an off-the-shelf 
lattice rule that hasn't been tailored to this problem, and so we observe nearly the optimal rate
but the constant may still depend on the dimension (which is fixed for these experiments).
For the other methods we observe the expected rates, with the exception of QMC,
which appears to not yet be in the asymptotic regime. Results for $\widetilde{p} = 3$
are very similar to those for $\widetilde{p} = 2$, and so have been omitted.

\subsection{Problem 2: Domain with interior islands}
Consider again the domain $D = (0, 1)^2$, and the subdomain consisting of four \emph{islands}
given by
$
\Df \coloneqq [\tfrac{1}{8}, \tfrac{3}{8}]^2 
\cup [\tfrac{5}{8}, \tfrac{7}{8}]^2
\cup [\tfrac{1}{8}, \tfrac{3}{8}] \times [\tfrac{5}{8}, \tfrac{7}{8}]
\cup [\tfrac{5}{8}, \tfrac{7}{8}] \times [\tfrac{1}{8}, \tfrac{3}{8}],$
see Figure~\ref{fig:islands} for a depiction.
Since we use uniform triangular  FE meshes with $h_\ell = 2^{-\ell + 3}$
the FE triangulation aligns with the boundaries of the components $\Df$
on all levels $\ell = 0, 1, 2, \ldots$. 
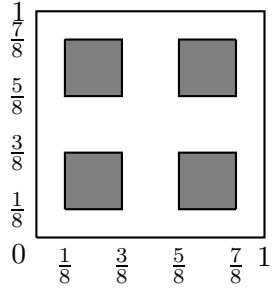
\begin{figure}[!h]
\centering
\begin{tikzpicture}[scale=3]
        \draw [thick] (0, 0) -- (1, 0) -- (1, 1) -- (0, 1) -- (0, 0);
        \fill[gray] (0.125, 0.125) rectangle (0.375, 0.375);
        \fill[gray] (0.625, 0.125) rectangle (0.875, 0.375);
        \fill[gray] (0.125, 0.625) rectangle (0.375, 0.875);
        \fill[gray] (0.625, 0.625) rectangle (0.875, 0.875);
        \draw [thick] (0.125, 0.125) -- (0.375, 0.125) -- (0.375, 0.375) -- (0.125, 0.375) -- (0.125, 0.125);
        \draw [thick] (1-0.125, 0.125) -- (1-0.375, 0.125) -- (1-0.375, 0.375) -- (1-0.125, 0.375) -- (1-0.125, 0.125);
        \draw [thick] (0.125, 1 - 0.125) -- (0.375, 1 - 0.125) -- (0.375, 1 - 0.375) -- (0.125, 1 - 0.375) -- (0.125, 1 - 0.125);
              \draw [thick] (1-0.125, 1 - 0.125) -- (1-0.375, 1 - 0.125) -- (1-0.375, 1 - 0.375) -- (1-0.125, 1 - 0.375) -- (1-0.125, 1 - 0.125);
        \draw (0, 1) node [left] {1};
        \draw(1, 0) node [below] {1};
        \draw(0, -0.07)  node [left] {0};
        \draw(0.125, 0) node [below] {$\tfrac{1}{8}$};
        \draw(0, 0.125) node [left] {$\tfrac{1}{8}$};
        \draw(0.375, 0) node [below] {$\tfrac{3}{8}$};
        \draw(0, 0.375) node [left] {$\tfrac{3}{8}$};
        \draw(1 - 0.125, 0) node [below] {$\tfrac{7}{8}$};
        \draw(0, 1 - 0.125) node [left] {$\tfrac{7}{8}$};
        \draw(1 - 0.375, 0) node [below] {$\tfrac{5}{8}$};
        \draw(0, 1 - 0.375) node [left] {$\tfrac{5}{8}$};        
        
\end{tikzpicture}
\caption{Domain $D$ with four islands forming $\Df$ (in grey).}
\label{fig:islands}
\end{figure}

The coefficients are now given by
\begin{alignat*}{2}
a_0(\bsx) &=
\begin{cases}
\sigdiff \coloneqq  0.01 & \text{if } \bsx \in \Df,\\[1mm]
\sigdiff' \coloneqq 0.011 &  \text{if } \bsx \in D \setminus\Df,
\end{cases}
\quad
&&a_j(\bsx) = 
\begin{cases}
\sigdiff w_{(j + 1)/2}(\widetilde{p}_a; \bsx) & \text{for $j$ odd},\\[1mm]
\sigdiff' w'_{j/2}(\widetilde{p}_a'; \bsx) & \text{for $j$ even},\\
\end{cases}
\\[1mm]
b_0(\bsx) &=\
\begin{cases}
\sigabs \coloneqq  2 & \text{if } \bsx \in \Df,\\[1mm]
\sigabs' \coloneqq 0.3 &  \text{if } \bsx \in D \setminus\Df,
\end{cases}
\quad
&&b_j(\bsx) =
\begin{cases}
\sigabs w_{(j + 1)/2}(\widetilde{p}_b; \bsx) & \text{for $j$ odd},\\[1mm]
\sigabs' w'_{j/2}(\widetilde{p}_b'; \bsx) & \text{for $j$ even},\\
\end{cases}
\end{alignat*}
where
\begin{align*}
w_{k}(q; \bsx) \,&=\, 
\begin{cases}
\frac{1}{k^{q}} \sin \big(8k \pi x_1\big) \sin\big(8(k + 1)\pi x_2\big)
& \text{for } \bsx \in \Df\,,\\
0 & \text{for } \bsx \in D \setminus \Df\,,\text{ and}
\end{cases}\\
w_{k}'(q; \bsx) \,&=\, \begin{cases}
0
& \text{for } \bsx \in \Df\,,\\
\frac{1}{k^{q}} \sin \big(8k\pi x_1\big) \sin\big(8(k + 1)\pi x_2\big) 
& \text{for } \bsx \in D \setminus \Df\,,
\end{cases}
\end{align*}
and where the parameters $\widetilde{p}_a, \widetilde{p}_a', \widetilde{p}_b, \widetilde{p}_b' \geq 4/3$ 
give the different decays of the coefficients on the different areas of the domain.
As for Problem 1, if any of $\widetilde{p}_a, \widetilde{p}_a', \widetilde{p}_b, \widetilde{p}_b'$
are less than 2, then we scale the corresponding zeroth term in the coefficient by $\pi/\sqrt{2}$.

The complexity of MLMC, MLQMC and the enhanced MLQMC using two-gird methods and nearby QMC points
for this problem is given in Figure~\ref{fig:cost2}. 
As expected for both MLQMC algorithms we observe a convergence rate of $-1$, and the MLMC results
approach the expected convergence rate of $-2$.
Also, since this problem is more difficult for eigensolvers to handle, 
we now observe that the two-grid MLQMC gives a speedup by a factor of more than 3.
Other tests using different values of $\widetilde{p}_a, \widetilde{p}_a', \widetilde{p}_b, \widetilde{p}_b'$
yielded similar results. Note also that numerical results for single level QMC methods applied 
to this problem (with slightly different $a_j$, $b_j$)
were given previously in \cite{GGKSS19}.
\begin{figure}[!t]
\includegraphics[scale=0.36]{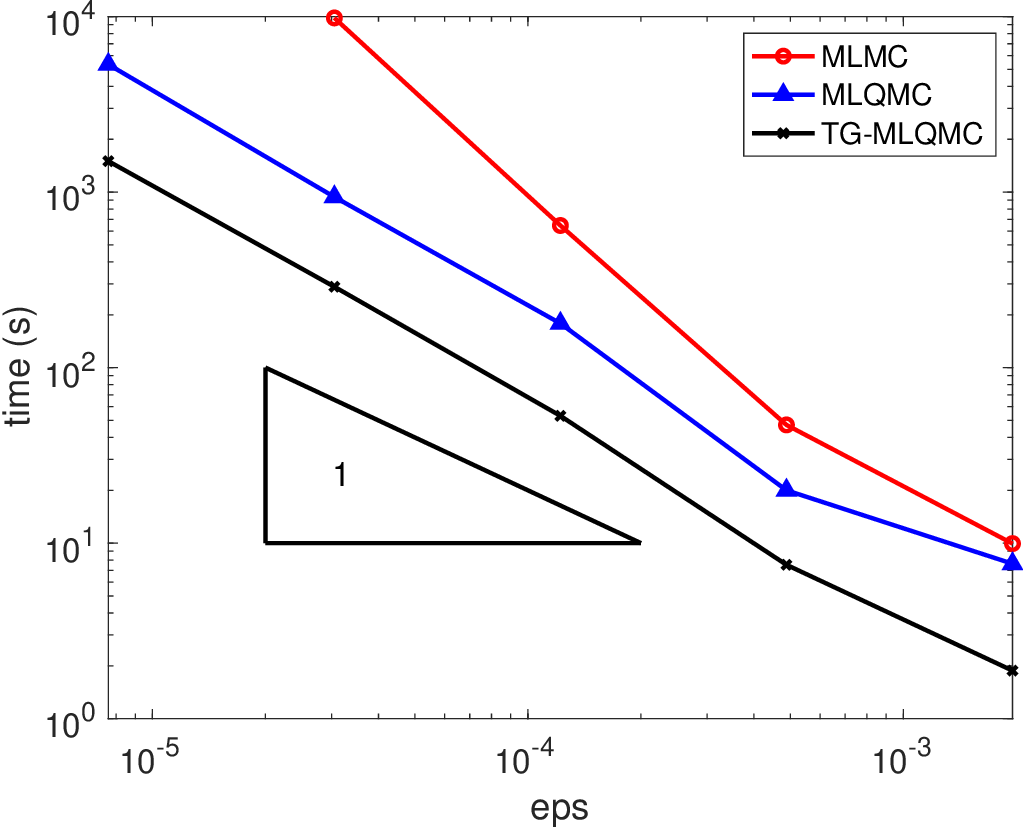}
\hfill
\includegraphics[scale=0.36]{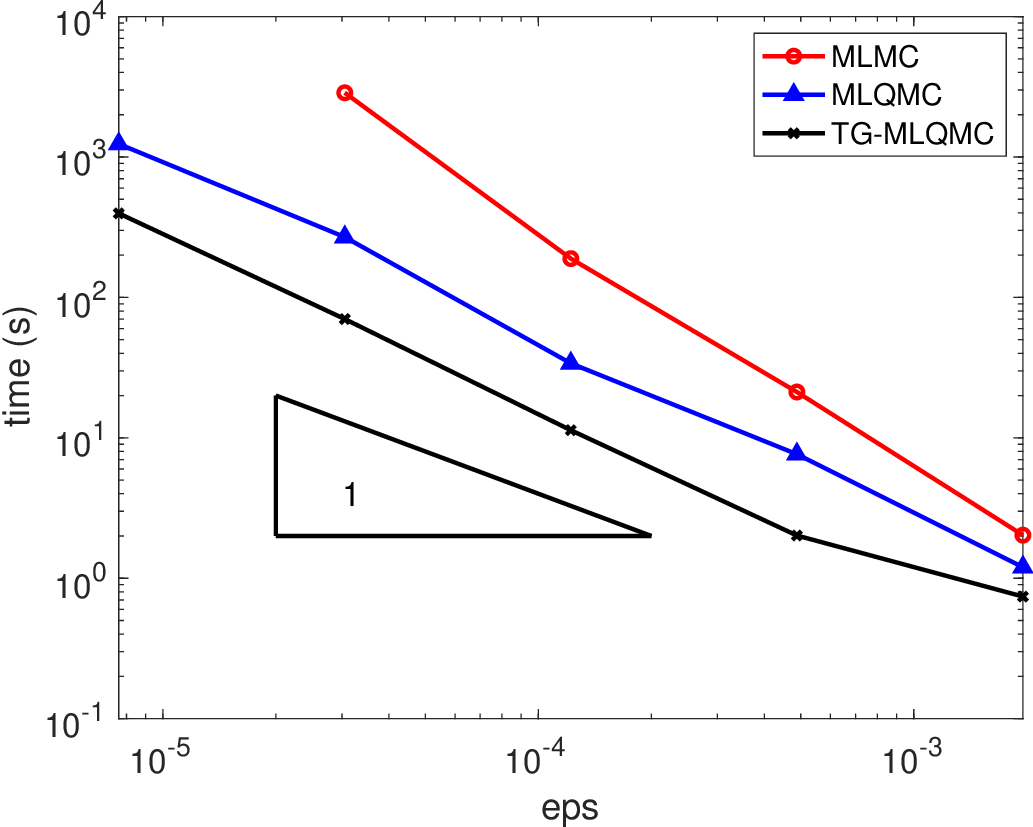}
\vspace{-3mm}
\caption{Problem 2: Complexity (measured as time in seconds)
of MLMC, plain-vanilla MLQMC and the enhanced MLQMC method
using the two-grid method and nearby QMC points
for 
$\widetilde{p}_a = \widetilde{p}_b = 4/3$, $\widetilde{p}_a' =\widetilde{p}_b' = 2$ (left), and
$\widetilde{p}_a = \widetilde{p}_a' = \widetilde{p}_b = \widetilde{p}_b' = 2$ (right). 
\label{fig:cost2}}
\end{figure}

\section{Conclusion}
We have developed an efficient MLQMC algorithm for random elliptic EVPs, which uses two-grid
methods and nearby QMC points to obtain a speedup compared to an ordinary MLQMC implementation.
We provided theoretical justification for the use of both strategies.
Finally, we presented numerical results for two test problems, which validate the theoretical results 
from the accompanying paper \cite{GS20a} and also demonstrate the speedup obtained by our new MLQMC 
algorithm.

\medskip
\noindent\textbf{Acknowledgements.} 
This work is supported by the Deutsche Forschungsgemeinschaft (German Research Foundation) 
under Germany’s Excellence Strategy EXC 2181/1 - 390900948 (the Heidelberg STRUCTURES 
Excellence Cluster). Also, this research includes computations using the computational 
cluster Katana supported by Research Technology Services at UNSW Sydney.

\small

\bibliographystyle{plain}
\bibliography{mlqmc-evp}

\begin{thebibliography}{10}

\bibitem{AS12}
R.~Andreev and Ch. Schwab.
\newblock {Sparse tensor approximation of parametric eigenvalue problems}.
\newblock In {I.~G.~Graham et~al.}, editor, {\em {Numerical Analysis of
  Multiscale Problems, Lecture Notes in Computational Science and
  Engineering}}, pages 203--241. Springer, Berlin, 2012.

\bibitem{AI10}
M.~N. Avramova and K.~N. Ivanov.
\newblock Verification, validation and uncertainty quantification in
  multi-physics modeling for nuclear reactor design and safety analysis.
\newblock {\em Prog. Nucl. Energy}, 52:601–--614, 2010.

\bibitem{AEHW12}
D.~A.~F. Ayres, M.~D. Eaton, A.~W. Hagues, and M.~M.~R. Williams.
\newblock Uncertainty quantification in neutron transport with generalized
  polynomial chaos using the method of characteristics.
\newblock {\em Ann. Nucl. Energy}, 45:14--–28, 2012.

\bibitem{BO91}
I.~Babu\v{s}ka and J.~Osborn.
\newblock Eigenvalue problems.
\newblock In {P.~G.~Ciarlet and J.~L.~Lions}, editor, {\em Handbook of
  Numerical Analysis, Volume 2: Finite Element Methods (Part 1)}, pages
  641--787. Elsevier, Amsterdam, 1991.

\bibitem{Banerjee92}
U.~Banerjee.
\newblock A note on the effect of numerical quadrature in finite element
  eigenvalue approximation.
\newblock {\em Numer. Math.}, 61:145--152, 1992.

\bibitem{BarSchwZol11}
A.~Barth, Ch. Schwab, and N.~Zollinger.
\newblock {Multilevel Monte Carlo finite element method for elliptic PDEs with
  stochastic coefficients}.
\newblock {\em Numer. Math.}, 119:123--161, 2011.

\bibitem{BierChern16}
C.~Bierig and A.~Chernov.
\newblock {Estimation of arbitrary order central statistical moments by the
  multilevel Monte Carlo method}.
\newblock {\em Stoch. Partial Differ.}, 4:3--40, 2016.

\bibitem{ClifGilSchTeck11}
K.~A. Cliffe, M.~B. Giles, R.Scheichl, and A.~L. Teckentrup.
\newblock {Multilevel Monte Carlo methods and applications to PDEs with random
  coefficients}.
\newblock {\em Comput. Visual. Sci.}, 14:3--15, 2011.

\bibitem{CKN06}
R.~Cools, F.Y. Kuo, and D.~Nuyens.
\newblock Constructing embedded lattices rules for multivariate integration.
\newblock {\em SIAM J. Sci. Comp.}, 28:2162--2188, 2006.

\bibitem{DKLeGNS14}
J.~Dick, F.~Y. Kuo, Q.~T. {Le~Gia}, D.~Nuyens, and Ch. Schwab.
\newblock {Higher-order QMC Petrov--Galerkin discretization for affine
  parametric operator equations with random field inputs}.
\newblock {\em SIAM J. Numer. Anal.}, 52:2676--2702, 2014.

\bibitem{DKS13}
J.~Dick, F.~Y. Kuo, and I.~H. Sloan.
\newblock {High-dimensional integration: The quasi-Monte Carlo way}.
\newblock {\em Acta Numer.}, 22:133--288, 2013.

\bibitem{DP10}
J.~Dick and F.~Pillichshammer.
\newblock {\em {Digital Nets and Sequences: Discrepancy Theory and Quasi-Monte
  Carlo Integration}}.
\newblock Cambridge University Press, New York, NY, 2010.

\bibitem{D99}
D.~C. Dobson.
\newblock {An efficient method for band structure calculations in 2D photonic
  crystals}.
\newblock {\em J. Comput. Phys.}, 149:363--376, 1999.

\bibitem{DH76}
J.~J. Duderstadt and L.~J. Hamilton.
\newblock {\em {Nuclear Reactor Analysis}}.
\newblock John Wiley \& Sons, New York, NY, 1976.

\bibitem{ElmSu19}
H.~C. Elman and T.~Su.
\newblock {Low-rank solution methods for stochastic eigenvalue problems}.
\newblock {\em SIAM J. Sci. Comp.}, 41:{A2657--A2680}, 2019.

\bibitem{GhanGhosh07}
R.~Ghanem and D.~Ghosh.
\newblock Efficient characterization of the random eigenvalue problem in a
  polynomial chaos decomposition.
\newblock {\em Int. J. Numer. Meth. Engng}, 72:486--504, 2007.

\bibitem{GGKSS19}
A.~D. Gilbert, I.~G. Graham, F.~Y. Kuo, R.~Scheichl, and I.~H. Sloan.
\newblock {Analysis of quasi-Monte Carlo methods for elliptic eigenvalue
  problems with stochastic coefficients}.
\newblock {\em Numer. Math.}, 142:863--915, 2019.

\bibitem{GGSS20}
A.~D. Gilbert, I.~G. Graham, R.~Scheichl, and I.~H. Sloan.
\newblock Bounding the spectral gap for an elliptic eigenvalue problem with
  uniformly bounded stochastic coefficients.
\newblock In {D. Wood et al.}, editor, {\em {2018 MATRIX Annals}}, pages
  29--43. Springer, Cham, 2020.

\bibitem{GS20a}
A.~D. Gilbert and R.~Scheichl.
\newblock {Multilevel quasi-Monte Carlo methods for random elliptic eigenvalue
  problems I: Regularity and analysis}.
\newblock {\em {Preprint, arXiv:2010.01044}}, 2022.

\bibitem{Giles08a}
M.~B. Giles.
\newblock {Multilevel Monte Carlo path simulation}.
\newblock {\em Oper. Res.}, 56:607--617, 2008.

\bibitem{GilesWater09}
M.~B. Giles and B.~Waterhouse.
\newblock {Multilevel quasi-Monte Carlo path simulation}.
\newblock In {\em {Advanced Financial Modelling, Radon Series on Computational
  and Applied Mathematics}}, pages 165--181. De Gruyter, New York, 2009.

\bibitem{HakKaarLaak15}
H.~Hakula, V.~Kaarnioja, and M.~Laaksonen.
\newblock Approximate methods for stochastic eigenvalue problems.
\newblock {\em Appl. Math. Comput.}, 267:664--681, 2015.

\bibitem{Hein01}
S.~Heinrich.
\newblock {Multilevel Monte Carlo methods}.
\newblock In {\em Multigrid Methods, Vol. 2179 of Lecture Notes in Computer
  Science}, pages 58--67. Springer, Berlin, 2001.

\bibitem{Hoer03}
L.~H\"ormander.
\newblock {\em {The Analysis of Linear Partial Differential Operators I}}.
\newblock Springer, Berlin, 2003.

\bibitem{HuCheng11}
X.~Hu and X.~Cheng.
\newblock {Acceleration of a two-grid method for eigenvalue problems}.
\newblock {\em Math. Comp.}, 80:1287--1301, 2011.

\bibitem{HuCheng15_corr}
X.~Hu and X.~Cheng.
\newblock {Corrigendum to: ``Acceleration of a two-grid method for eigenvalue
  problems''}.
\newblock {\em Math. Comp.}, 84:2701--2704, 2015.

\bibitem{JC13}
E.~Jamelota and P.~{Ciarlet Jr}.
\newblock {Fast non-overlapping Schwarz domain decomposition methods for
  solving the neutron diffusion equation}.
\newblock {\em J. Comput. Phys.}, 241:445–--463, 2013.

\bibitem{Joe06}
S.~Joe.
\newblock Construction of good rank-1 lattice rules based on the weighted star
  discrepancy.
\newblock In {H. Niederreiter \& D. Talay}, editor, {\em {Monte Carlo and
  Quasi-Monte Carlo Methods 2004}}, pages 181--196. Springer, Berlin, 2006.

\bibitem{K01}
P.~Kuchment.
\newblock {The mathematics of photonic crystals}.
\newblock {\em {SIAM, Frontiers of Applied Mathematics}}, 22:207--272, 2001.

\bibitem{KuoLattice}
F.~Y. Kuo.
\newblock \url{https://web.maths.unsw.edu.au/~fkuo/lattice/index.html}.
\newblock {\em Accessed August 24, 2020}, 2007.

\bibitem{KSSSU17}
F.~Y. Kuo, R.~Scheichl, Ch. Schwab, I.~H. Sloan, and E.~Ullmann.
\newblock {Multilevel quasi-Monte Carlo methods for lognormal diffusion
  problems}.
\newblock {\em {Math.~Comp.}}, 86:2827--2860, 2017.

\bibitem{KSS15}
F.~Y. Kuo, Ch. Schwab, and I.~H.Sloan.
\newblock {Multi-level quasi-Monte Carlo finite element methods for a class of
  elliptic PDEs with random coefficients}.
\newblock {\em {Found.~Comput.~Math}}, 15:411--449, 2015.

\bibitem{KSS12}
F.~Y. Kuo, Ch. Schwab, and I.~H. Sloan.
\newblock {Quasi-Monte Carlo finite element methods for a class of elliptic
  partial differential equations with random coefficients}.
\newblock {\em {SIAM J.~Numer.~Anal.}}, 50:3351--3374, 2012.

\bibitem{NS12}
R.~Norton and R.~Scheichl.
\newblock Planewave expansion methods for photonic crystal fibres.
\newblock {\em {Appl.~Numer.~Math.}}, 63:88--104, 2012.

\bibitem{NC06}
D.~Nuyens and R.~Cools.
\newblock {Fast algorithms for component-by-component construction of rank-1
  lattice rules in shift-invariant reproducing kernel Hilbert spaces}.
\newblock {\em Math.~Comp.}, 75:903--920, 2006.

\bibitem{NC06np}
D.~Nuyens and R.~Cools.
\newblock {Fast component-by-component construction of rank-1 lattice rules
  with a non-prime number of points}.
\newblock {\em J.~Complexity}, 22:4--28, 2006.

\bibitem{Parlett80}
B.~N. Parlett.
\newblock {\em {The Symmetric Eigenvalue Problem}}.
\newblock Prentice--Hall, Englewood Cliffs, NJ, 1980.

\bibitem{QiuLyu20}
Z.~Qui and Z.~Lyu.
\newblock Vertex combination approach for uncertainty propagation analysis in
  spacecraft structural system with complex eigenvalue.
\newblock {\em Acta Astronaut.}, 171:106--117, 2020.

\bibitem{RobNuyVan19}
P.~Robbe, D.~Nuyens, and S.~Vandewalle.
\newblock {Recycling samples in he multigrid multilevel (quasi)-Monte Carlo
  method}.
\newblock {\em SIAM J.~Sci.~Comp.}, 41:S37--S60, 2019.

\bibitem{Saad11}
Y.~Saad.
\newblock {\em {Numerical Methods for Large Eigenvalue Problems}}.
\newblock SIAM, Philadelphia, PA, 2011.

\bibitem{ShinAst72}
M.~Shinozuka and C.~J. Astill.
\newblock Random eigenvalue problems in structural analysis.
\newblock {\em AIAA Journal}, 10:456--462, 1972.

\bibitem{SW98}
I.~H. Sloan and H.~Wo\'zniakowski.
\newblock When are quasi-monte carlo algorithms efficient for high dimensional
  integrals?
\newblock {\em J. Complexity}, 14:1--33, 1998.

\bibitem{Thom81}
W.~T. Thomson.
\newblock {\em {The Theory of Vibration with Applications}}.
\newblock Prentice--Hall, Englewood Cliffs, NJ, 1981.

\bibitem{W10}
M.~M.~R. Williams.
\newblock A method for solving stochastic eigenvalue problems.
\newblock {\em Appl.~Math.~Comput.}, 215:4729--–4744, 2010.

\bibitem{Xie14}
H.~Xie.
\newblock {A multigrid method for eigenvalue problem}.
\newblock {\em J.~Comput.~Phys.}, 274:550--561, 2014.

\bibitem{XuZhou99}
J.~Xu and A.~Zhou.
\newblock {A two-grid disretization scheme for eigenvalue problems}.
\newblock {\em Math.~Comp.}, 70:17--25, 1999.

\bibitem{YangBi11}
Y.~Yang and H.~Bi.
\newblock {Two-grid finite element discretization schemes based on
  shifted-inverse power method for elliptic eigenvalue problems}.
\newblock {\em SIAM J.~Numer.~Anal.}, 49:1602--1624, 2011.

\end{thebibliography}

\begin{appendix}
\section{Proof of Theorem~\ref{thm:2grid}}
\begin{proof}
First of all, we can use the triangle inequality to split the eigenfunction error into
\begin{align}
\nonumber
\|u(\bsy) - u^{h, s}(\bsy)\|_V \,&\leq\, 
\|u(\bsy) - u_s(\bsy)\|_V + \|u_s(\bsy) - u_{h, s}(\bsy)\|_V
+ \|u_{h, s}(\bsy) - u^{h, s}(\bsy)\|_V\\
&\lesssim\, 
 h +  s^{-(1/p + 1)}
+ \|u_{h, s}(\bsy) - u^{h, s}(\bsy)\|_V,
\label{eq:tg_u_tri}
\end{align}
where we have used \cite[Theorem~4.1]{GGKSS19} and \eqref{eq:fe_lam_u}, and
the constant is independent of $h$, $s$ and $\bsy$.  
All that remains for the eigenfunction result is to bound the third term above.

To this end, we can rewrite Step 2 of Algorithm~\ref{alg:2grid} using \eqref{eq:T_h,s} as
\[
\calA_s(\bsy; u^{h, s}(\bsy) - \lambda_{H, S}(\bsy) T_{h, s} u^{h, s}(\bsy), v_h)
\,=\, \calA_s(\bsy; T_{h, s} u_{H, S}(\bsy), v_h)
\quad \text{for all } v_h \in V\,,
\]
which is equivalent to the operator equation: Find $u^{h, s}(\bsy) \in V_h$
such that
\[
\bigg(\frac{1}{\lambda_{H, S}(\bsy)}  - T_{h, s}\bigg) u^{h, s}(\bsy) \,=\, 
\frac{1}{\lambda_{H, S}(\bsy)} T_{h, s} u_{H, S}(\bsy).
\]
This is in turn equivalent (up to a constant scaling factor) to the problem:
find $\widetilde{u} \in V_h$ such that
\begin{equation}
\label{eq:two_grid_op_eq}
\bigg(\frac{1}{\lambda_{H,S}(\bsy)} - T_{h, s}\bigg)\widetilde{u} \,=\, 
\frac{\lambda_{H, S}(\bsy) T_{h, s}u_{H, S}(\bsy)}
{\nrm{\lambda_{H, S}(\bsy) T_{h, s}u_{H, S}(\bsy)}{V}}
\,\eqqcolon\, u_0\,.
\end{equation}
Explicitly,
\[
\widetilde{u} \,=\, \frac{\lambda_{H, S}(\bsy) u^{h, s}(\bsy)}{\|T_{h, s} u_{H, S}(\bsy)\|_V},
\]
but after normalisation (Step 3) $u^{h, s}(\bsy) = \widetilde{u}/ \nrm{\widetilde{u}}{\calM}$.

We now apply Theorem~3.2 from \cite{YangBi11} to \eqref{eq:two_grid_op_eq},
using the space $X = V$ and with
\[
\mu_0 = \frac{1}{\lambda_{H, S}(\bsy)}
\quad \text{and} \quad 
u_0 \,=\, \frac{\lambda_{H, S}(\bsy) T_{h, s} u_{H, S}(\bsy)}{\|\lambda_{H, S}(\bsy) T_{h, s} u_{H, S}(\bsy)\|_V}.
\]
To do so, we must first verify that $\mu_0$ and $u_0$ satisfy the
required assumptions of \cite[Theorem~3.2]{YangBi11}, 
namely, $\|u_0\|_V = 1$, $\mu_0$ is not an eigenvalue of $T_{h, s}$,
and for all $\bsy \in \Omega$
\begin{align}
\label{eq:YangBi1}
\dist(u_0, E_h(\lambda_s(\bsy))  \,&\leq\, \frac{1}{2} \quad \text{and}\\
\label{eq:YangBi2}
|\mu_0 - \mu_{2, h, s}(\bsy)| \,&\geq\, \frac{\mu_{1, s}(\bsy) - \mu_{2, s}(\bsy)}{2}
\,\eqqcolon\,\frac{\widetilde{\rho}_s(\bsy)}{2}.
\end{align}
Recall that $\mu_k(\bsy) = 1/\lambda_k(\bsy)$ is an
eigenvalue of $T$, and similarly, subscripts $h$ and $s$ denote their FE and 
dimension-truncated counterparts, respectively. 
Clearly, the first two assumptions hold, 
and so it remains to verify \eqref{eq:YangBi1} and \eqref{eq:YangBi2}.

To show \eqref{eq:YangBi1}, since $\lambda_{s}(\bsy)$ is simple we have
\begin{align}
\label{eq:u0dist}
\dist (u_0, &E_h(\lambda_s(\bsy)) 
\nonumber\\&
=\, \inf_{\alpha \in \R} \|u_0 - \alpha u_{h, s}(\bsy)\|_V
\nonumber\\
&=\, \frac{1}{\lambda_{H, S}(\bsy) \|T_{h, s} u_{H, S}(\bsy)\|_V}
\inf_{\alpha \in \R} \|\lambda_{H, S}(\bsy) T_{h, s} u_{H, S}(\bsy) - \alpha u_{h, s}(\bsy)\|_V.
\end{align}

To show that the first factor can be bounded by a constant, we use 
the reverse triangle inequality along with the lower bound
\eqref{eq:lam_bnd}, which since $H>0$ was assumed to be sufficiently small gives
\begin{equation}
\label{eq:lamTu_tri}
\lambda_{H, S}(\bsy) \|T_{h, s} u_{H, S}(\bsy)\|_V \,\geq\, \lambdaunder
\big| \|Tu_{H, S}(\bsy)\|_V - \|(T - T_{h, s})u_{H, S}(\bsy)\|_V\big|.
\end{equation}
Now, by the equivalence of norms \eqref{eq:A_equiv} we have
\begin{align*}
\|T u_{H, S}(\bsy)\|_V \,\geq\, \frac{1}{C_\calA}
\sqrt{\calA(\bsy; T u_{H, S}(\bsy), T u_{H, S}(\bsy))}.
\end{align*}
Then using the definition of $T$, along with the facts
that $u_{H, S}(\bsy)$ is an eigenfunction and $\calA(\bsy)$ is symmetric,
we can simplify this as
\begin{align*}
\calA(\bsy; T u_{H, S}(\bsy), T u_{H, S}(\bsy)) 
\,&=\, \calM(u_{H, S}(\bsy), T u_{H, S}(\bsy) )
\\
&=\, \frac{1}{\lambda_{H, S}(\bsy)} \calA(\bsy; u_{H, S}(\bsy), T u_{H, S}(\bsy))
\\
&=\, \frac{1}{\lambda_{H, S}(\bsy)} \calM(u_{H, S}(\bsy), u_{H, S}(\bsy))
\,\geq\, \frac{1}{\lambdaover},
\end{align*}
where for the last inequality we have used \eqref{eq:lam_bnd} and $\|u_{H, S}(\bsy)\|_\calM = 1$. 
Hence, we have the constant lower bound
\begin{equation}
\label{eq:Tu_lower}
\|T u_{H, S}(\bsy)\|_V \,\geq\, C_\calA^{-1} \lambdaover^{-1/2},
\end{equation}
which is independent of $\bsy$, $S$ and $H$.

For the second term in \eqref{eq:lamTu_tri}, by \eqref{eq:T-T_hs}
we have the upper bound
\begin{align}
\label{eq:Tu-T_hs_u}
\|(T-T_{h, s}) u_{H, S}(\bsy)\| \,&=\, \|T - T_{h, s}\| \|u_{H, S}(\bsy)\|_V
\nonumber\\
&\leq\, C_T(s^{-1/p + 1} + h) \sqrt{\lambda_{H, S}(\bsy)} \|u_{H, S}(\bsy)\|_{\calM}
\nonumber\\
&\leq\, \lambdaover^{1/2} C_T (s^{-1/p + 1} + h),
\end{align}
where we have used that $u_{H, S}(\bsy)$ is an eigenfunction, normalised in $\calM$,
and also \eqref{eq:lam_bnd}.

Returning to \eqref{eq:lamTu_tri}, since $\|T u_{H, S}(\bsy)\|_V$
is bounded from below by a constant, by \eqref{eq:Tu-T_hs_u}
there exists $S_0 \in \N$ sufficiently large and $H_0 > 0$ sufficiently small
such that for all $s \geq S_0$ and $h \leq H_0$
we have $\|T u_{H, S}(\bsy)\|_V > \|(T - T_{h, s})u_{H, S}(\bsy)\|_V$.
Thus, substituting \eqref{eq:Tu_lower} and \eqref{eq:Tu-T_hs_u} into \eqref{eq:lamTu_tri}, 
we have the lower bound
\begin{align*}
\lambda_{H, S}(\bsy)\|T_{h, s}u_{H, S}(\bsy)\|_V \,&\geq\,
C_\calA^{-1} \lambdaunder\lambdaover^{-1/2}  
-  \lambdaunder \lambdaover^{1/2} C_T (s^{-1/p + 1} + h)
\\
&\geq\, 
C_\calA^{-1}\lambdaunder\lambdaover^{-1/2} 
- \lambdaunder\lambdaover^{1/2} C_T (S_0^{-1/p + 1} + H_0)
 \,\eqqcolon\, \frac{1}{C_{u_0}} \,>\, 0,
\end{align*}
where $0 < C_{u_0} < \infty $ is independent of $s, S, h, H$ and $\bsy$.
It follows that
\[
\dist\big(u_0, E_h(\lambda_s(\bsy))\big)
\,\leq\, C_{u_0} \|u_{h, s}(\bsy) - \lambda_{H, S}(\bsy) T_{h, s} u_{H, S}(\bsy)\|_V.
\]
For the second factor in \eqref{eq:u0dist}, using \eqref{eq:T_h,s},
for all $v_h \in V_h$ we have the identity
\begin{align*}
\calA_s(\bsy; u_{h, s}(\bsy) - \lambda_{H, S}T_{h, s} u_{H, S}(\bsy), v_h)
\,=\,
&\big[\lambda_{h, s}(\bsy) - \lambda_{H, S}(\bsy)\big] \calM(u_{H, S}(\bsy), v_h)
\\
&+ \lambda_{h, s}(\bsy) \calM(u_{h, s}(\bsy) - u_{H, S}(\bsy), v_h)\,.
\end{align*}
Letting $v_h = u_{h, s}(\bsy) - \lambda_{H, S}(\bsy) T_{h, s} u_{H, S}(\bsy)$, then using that
$\calA_s$ is coercive, as well as applying the triangle and Cauchy--Schwarz inequalities, we have
\begin{align*}
\|u_{h, s}(\bsy) -  &\lambda_{H, S}T_{h, s} u_{H, S}(\bsy)\|_V^2\\
&\lesssim\, 
\big|\lambda_{h, s}(\bsy) - \lambda_{H, S}(\bsy)\big| \|u_{H, S}\|_\calM
\|u_{h, s}(\bsy) - \lambda_{H, S}(\bsy) T_{h, s} u_{H, S}(\bsy)\|_\calM\\
&\quad+ \lambda_{h, s}(\bsy) \nrm{u_{h, s}(\bsy) - u_{H, S}(\bsy)}{\calM}
\|u_{h, s}(\bsy) - \lambda_{H, S}(\bsy) T_{h, s} u_{H, S}(\bsy)\|_\calM.
\end{align*}
Dividing through by $\|u_{h, s}(\bsy) - \lambda_{H, S}(\bsy) T_{h, s}
u_{H, S}(\bsy)\|_V$ and applying the Poincar\'e inequality \eqref{eq:poin} gives
\[
\|u_{h, s}(\bsy) -  \lambda_{H, S}T_{h, s} u_{H, S}(\bsy)\|_V\\
\,\lesssim\, \big|\lambda_{h, s}(\bsy) - \lambda_{H, S}(\bsy)\big| 
+ \lambdaover \nrm{u_{h, s}(\bsy) - u_{H, S}(\bsy)}{\calM},
\]
where we have also used that $\|u_{H, S}(\bsy)\|_\calM = 1$ and \eqref{eq:lam_bnd}. 
We can incorporate $\lambdaover$ into the constant, and then split the
right hand side again using the triangle inequality to give
\begin{align*}
\|u_{h, s}(\bsy) -  &\lambda_{H, S}T_{h, s} u_{H, S}(\bsy)\|_V
\,\lesssim\, \big|\lambda(\bsy) - \lambda_s(\bsy)| 
+ |\lambda_s(\bsy) - \lambda_{h, s}(\bsy)|
\\
&+  |\lambda(\bsy) - \lambda_S(\bsy)| 
+ |\lambda_S(\bsy) - \lambda_{H, S}(\bsy)\big| 
+ \|u(\bsy) - u_s(\bsy)\|_V 
\\
&+ \|u_s(\bsy) - u_{h, s}(\bsy)\|_\calM 
+ \|u(\bsy) - u_S(\bsy)\|_V 
+ \|u_S(\bsy) - u_{H, S}(\bsy)\|_\calM,
\end{align*}
where we have also applied the Poincar\'e inequality \eqref{eq:poin}
again to switch to the $V$-norms
for the eigenfunction truncation errors.

Now, each of the terms in \eqref{eq:u0dist} can be bounded by \cite[Theorems 2.6 \& 4.1]{GGKSS19}
to give
\begin{align}
\label{eq:tg_dist_upper}
\dist(u_0, E_h(\lambda_s(\bsy)) \,&\lesssim\,
\|u_{h, s}(\bsy) -  \lambda_{H, S}T_{h, s} u_{H, S}(\bsy)\|_V 
\nonumber\\
&\lesssim\,
s^{-1/p + 1} + S^{-1/p + 1} + h^2 + H^2,
\end{align}
where to bound the FE error in the $\calM$-norm we have used \cite[eqn.~(2.35)]{GGKSS19} 
with the functional 
$\calG = \calM(\cdot, u_s(\bsy) - u_{h, s}(\bsy))/\|u_s(\bsy) - u_{h, s}(\bsy)\|_\calM \in L^2(D)$
(and similarly for $u_S(\bsy) - u_{H, S}(\bsy)$).
It follows from \eqref{eq:tg_dist_upper} 
that there exists $S$ sufficiently large and $H$
sufficiently small --- both independent of $\bsy$ --- such that
\eqref{eq:YangBi1} holds.

Next, to verify \eqref{eq:YangBi2}, since
$\mu_0 = 1/\lambda_{H, S}(\bsy) \eqqcolon \mu_{H, S}(\bsy)$ and since the FE eigenvalues converge
from above and thus $\mu_{2,h, s}(\bsy) \leq \mu_{2, s}(\bsy)$, 
\begin{align}
\label{eq:tg_mu0_a}
|\mu_0 - \mu_{2, h, s}(\bsy)| \,&=\, \mu_{H, S}(\bsy) - \mu_{2, h, s}(\bsy)
\,\geq\, \mu_{H, S}(\bsy) - \mu_{2, s}(\bsy)
\nonumber\\
&=\, \widetilde{\rho}_s(\bsy) - \big(\mu_s(\bsy) - \mu_{H, S}(\bsy)\big).
\end{align}

Now, suppose that $\big(\mu_s(\bsy) - \mu_{H, S}(\bsy)\big) \leq 0$, then
\eqref{eq:tg_mu0_a} simplifies to 
\[
|\mu_0 - \mu_{2, h, s}(\bsy)| \,\geq\,
\widetilde{\rho}_s(\bsy) \,\geq\, \frac{\widetilde{\rho}_s(\bsy)}{2},
\]
as required. Alternatively, if $\big(\mu_s(\bsy) - \mu_{H, S}(\bsy)\big) > 0$
then \eqref{eq:tg_mu0_a} becomes
\[
|\mu_0 - \mu_{2, h, s}(\bsy)| \,\geq\,
\widetilde{\rho}_s(\bsy) - \big|\mu_s(\bsy) - \mu_{H,S}(\bsy)\big|.
\]
By the triangle inequality we can bound the second term on the right, 
again using the bounds from \cite[Theorems~2.6 \& 4.1]{GGKSS19}, as
well as \eqref{eq:lam_bnd}, to give
\begin{align*}
\big|\mu_s(\bsy) - \mu_{H, S}(\bsy)\big| 
\,&\leq\, \frac{|\lambda(\bsy) - \lambda_s(\bsy)|}{\lambda(\bsy)\lambda_s(\bsy)} 
+ \frac{|\lambda(\bsy) - \lambda_S(\bsy)|}{\lambda(\bsy)\lambda_S(\bsy)}
+ \frac{|\lambda_S(\bsy) - \lambda_{H, S}(\bsy)|}{\lambda_{S}(\bsy) \lambda_{H, S}(\bsy)}
\\
&\leq\, \frac{C}{\lambdaunder^2}(s^{-1/p + 1} + S^{-1/p + 1} +H^2).
\end{align*}
The upper bound is independent of $\bsy$, thus we can take 
$S$ sufficiently large and $H$ sufficiently small,
such that, using the bound on the spectral gap 
in \eqref{eq:gap} together with \eqref{eq:lam_bnd},
\begin{equation}
\label{eq:tg_mu0_b}
\big|\mu_s(\bsy) - \mu_{H, S}(\bsy)\big| \,\leq\, 
\frac{1}{2}\frac{\rho}{\overline{\lambda_1}\overline{\lambda_2}}
\,\leq\, \frac{1}{2}\frac{\lambda_{2, s}(\bsy) - 
\lambda_{s}(\bsy)}{\lambda_s(\bsy)\lambda_{2, s}(\bsy)}
\,=\, \frac{\widetilde{\rho}_s(\bsy)}{2}.
\end{equation}
Then, to show \eqref{eq:YangBi2} we can substitute the bound above into
\eqref{eq:tg_mu0_a}.

Hence, we have verified the assumptions for \cite[Theorem~3.2]{YangBi11} for all $\bsy$.
Since $\lambda_s(\bsy)$, $\lambda_{h, s}(\bsy)$ are simple,
$\dist(u^{h, s}(\bsy), \Ehat_h(\lambda_s(\bsy)) = \|u_{h, s}(\bsy) - u^{h, s}(\bsy)\|_V$
and hence, it now follows from \cite[Theorem~3.2]{YangBi11} 
that
\begin{equation}
\label{eq:u_s_YangBi}
\|u_{h, s}(\bsy) - u^{h, s}(\bsy)\|_V 
\leq\, \frac{16}{\widetilde{\rho}_s(\bsy)} 
\frac{|\lambda_{h, s}(\bsy) - \lambda_{H, S}(\bsy)|}{\lambda_{h, s}(\bsy)\lambda_{H, S}(\bsy)}
\|u_{h, s}(\bsy) - \lambda_{H, S}(\bsy)T_{h, s} u_{H, S}(\bsy)\|_V.
\end{equation}

We handle the three factors in turn. For the first factor, by the argument used in \eqref{eq:tg_mu0_b}
we have $1/\widetilde{\rho}_s(\bsy) \leq \overline{\lambda_1} \overline{\lambda_2}/\rho$,
independently of $\bsy$.
For the second factor, we can use the 
uniform lower bound \eqref{eq:lam_bnd}, and then the triangle inequality to give the 
upper bound
\begin{align*}
\frac{|\lambda_{h, s}(\bsy) - \lambda_{H, S}(\bsy)|}{\lambda_{h, s}(\bsy)\lambda_{H, S}(\bsy)}
\leq\, \frac{1}{\underline{\lambda}^2} 
\Big(&|\lambda(\bsy) - \lambda_s(\bsy)| + |\lambda_s(\bsy) - \lambda_{h, s}(\bsy)| \\
&+ 
|\lambda(\bsy) - \lambda_S(\bsy)| + |\lambda_S(\bsy) - \lambda_{H, S}(\bsy)|\Big)\,.
\end{align*}
Each term above can be bounded by using one of Theorems 2.6 or 4.1 from \cite{GGKSS19}
to give
\begin{equation}
\label{eq:tg_u_s1}
\frac{|\lambda_{h, s}(\bsy) - \lambda_{H, S}(\bsy)|}{\lambda_{h, s}(\bsy)\lambda_{H, S}(\bsy)}
 \,\lesssim\, 
s^{-(1/p - 1)} + h^2 + S^{-(1/p - 1)} + H^2,
\end{equation}
where the constant is again independent of $s, S, h, H$ and $\bsy$.

Finally, the third factor in \eqref{eq:u_s_YangBi} can be bounded using \eqref{eq:tg_dist_upper}.
Hence, substituting \eqref{eq:tg_u_s1} and \eqref{eq:tg_dist_upper} into \eqref{eq:u_s_YangBi}
we obtain the upper bound
\begin{align}
\label{eq:tg_u_hs_err}
\|u_{h, s}(\bsy) - u^{h, s}(\bsy)\|_V 
\,&\lesssim\, 
s^{-2(1/p - 1)} + h^4 + S^{-2(1/p - 1)} + H^4
\nonumber\\
+ 2 \big(s^{-(1/p - 1)}h^2 + 
& s^{-(1/p - 1)} S^{-(1/p - 1)} + s^{-(1/p - 1)}H^2
+ h^2 S^{-(1/p - 1)} + h^2 H^2 + S^{-(1/p - 1)} H^2 \big)
\nonumber\\
&\lesssim\,
H^4 + S^{-2(1/p - 1)} + H^2 S^{-(1/p - 1)}\,,
\end{align}
where we have used the fact that $s \geq S$ and $h \leq H$ to obtain the last inequality.
Then to give the error bound \eqref{eq:tg_u_err}, 
we simply substitute \eqref{eq:tg_u_hs_err} into \eqref{eq:tg_u_tri}.

The second result \eqref{eq:tg_lam_err} follows from
Lemma~\ref{lem:RQ-diff}, by choosing 
$\calB = \calA(\bsy; \cdot, \cdot)$, 
$\widetilde{\calB} = \calA_s(\bsy; \cdot, \cdot) = \calA(\bsy_s; \cdot, \cdot)$, 
$u = u(\bsy)$ and $w = u^{h, s}(\bsy)$. Noting that $\|u^{h,
  s}\|_\calM = 1$ and 
using the definition of $\lambda^{h, s}(\bsy)$ in
\eqref{eq:RQ_update},  this gives
\begin{align}
\lambda^{h, s}(\bsy)  - \lambda(\bsy)
\,=\, &\|u(\bsy) - u^{h, s}(\bsy)\|_{\calA(\bsy_s)} ^2
- \lambda(\bsy) \|u(\bsy) - u^{h, s}(\bsy)\|_{\calA(\bsy)}^2
\nonumber\\
&+ \calA\big(\bsy - \bsy_s; u(\bsy), u(\bsy) - 2u^{h, s}(\bsy)\big)
\nonumber\\
\lesssim\, & \|u(\bsy) - u^{h, s}(\bsy)\|_V^2 
+ \calA\big(\bsy - \bsy_s; u(\bsy), u(\bsy) - 2u^{h, s}(\bsy)\big),
\label{eq:lam-lam^hs=}
\end{align}
where we simplified using the linearity of $\calA(\bsy)$ in $\bsy$ and
used the equivalence of norms in \eqref{eq:A_equiv} and \eqref{eq:lam_bnd}, 
which both hold for all $\bsy$.

The last term from \eqref{eq:lam-lam^hs=} is bounded as follows
\begin{align}
\calA(\bsy - \bsy_s; &u(\bsy),  u(\bsy) - 2u^{h, s}(\bsy))
\nonumber\\
&=\, \int_D \sum_{j > s} \big(y_j a_j(\bsx) \nabla u(\bsy) \cdot \nabla [u(\bsy) - 2u^{h, s}(\bsy)]
\nonumber\\
&\qquad\qquad+ y_j b_j(\bsx) u(\bsy)[u(\bsy) - 2u^{h, s}(\bsy)]
\big) \rd \bsx
\nonumber\\
&\leq\, \frac{1}{2} \sum_{j > s} \big[
\|a_j\|_{L^\infty} \| u(\bsy) \|_V (\| u(\bsy) \|_V + 2 \|u^{h, s}(\bsy)\|_V)
\nonumber\\
&\qquad\qquad+ \| b_j \|_{L^\infty} \| u(\bsy) \|_{L^2} (\| u(\bsy) \|_{L^2} + 2 \|u^{h, s}(\bsy)\|_{L^2})\big]
\nonumber\\
&\lesssim\, \sum_{j > s} \max \big( \| a_ j\|_{L^\infty}, \| b_j \|_{L^\infty}\big) 
\,\lesssim\, s^{-(1/p - 1)},
\label{eq:A-A_s}
\end{align}
where in the second last inequality we have bounded the $V$-norms using 
\eqref{eq:u_bnd} and the $L^2$-norms using 
the equivalence to the $\calM$-norm \eqref{eq:M_equiv}, and then used
that $u(\bsy)$ and $u^{h, s}(\bsy)$ are normalised. The tail sum in the last inequality is 
bounded using \cite[Theorem 5.1]{KSS12}.

Finally, the result \eqref{eq:tg_lam_err} is obtained by substituting 
\eqref{eq:tg_u_err} and \eqref{eq:A-A_s} into \eqref{eq:lam-lam^hs=}.
\end{proof}
\end{appendix}

\end{document}